\documentclass[11pt]{article}

\usepackage{tikz}
\usepackage[utf8]{inputenc}
\usepackage[T1]{fontenc}
\usepackage{graphicx}
\usepackage{xcolor}
\usepackage[english]{babel}
\usepackage{hyperref}
\usepackage{fancyhdr}
\usepackage{amsmath, amstext, amsopn, amssymb,amsthm}
\usepackage{thmtools}
\usepackage{mathtools}
\usepackage{stmaryrd}
\SetSymbolFont{stmry}{bold}{U}{stmry}{m}{n}
\usepackage{bm}
\usepackage{bbm}
\usepackage[shortlabels, inline]{enumitem}
\usepackage[left=3cm,right=3cm,top=3cm,bottom=3cm]{geometry}
\usepackage[nopatch=footnote]{microtype}

\addto\extrasenglish{

}
\hypersetup{
	colorlinks   = true,
	urlcolor     = blue,
	linkcolor    = blue,
	citecolor    = blue
}

\DeclareMathOperator*{\argmin}{arg\,min}

\DeclareMathOperator{\supp}{supp}

\DeclareMathOperator{\id}{Id}
\DeclareMathOperator{\spanop}{span}

\DeclarePairedDelimiter{\norm}{\lVert}{\rVert}
\DeclarePairedDelimiter{\abs}{\lvert}{\rvert}
\DeclarePairedDelimiterX{\inner}[2]{\langle}{\rangle}{{#1},{#2}}
\DeclarePairedDelimiter{\bbrac}{\llbracket}{\rrbracket}

\newcommand{\R}{\mathbb{R}}
\newcommand{\N}{\mathbb{N}}

\newcommand{\OT}{\mathcal{T}}

\newcommand{\calN}{\mathcal{N}}

\newcommand{\eps}{\varepsilon}
\newcommand{\sym}{\operatorname{Sym}}
\newcommand{\bmpi}{\bm{\pi}}
\newcommand{\bmalpha}{\bm{\alpha}}

\newcommand{\bmu}{\bm{u}}
\DeclareMathOperator{\OPT}{OTP}
\DeclareMathOperator{\EP}{EP}
\DeclareMathOperator{\LP}{LP}
\DeclareMathOperator{\coh}{coh}
\newcommand{\vc}[1]{\bm{#1}}
\DeclareMathOperator{\shiftc}{SC}

\DeclareMathOperator{\im}{im}
\DeclareMathOperator{\aff}{aff}

\makeatletter
\newcommand*{\bigcdot}{}%
\DeclareRobustCommand*{\bigcdot}{%
	\mathbin{\mathpalette\bigcdot@{}}%
}
\newcommand*{\bigcdot@scalefactor}{.5}
\newcommand*{\bigcdot@widthfactor}{1.15}
\newcommand*{\bigcdot@}[2]{%
	\sbox0{$#1\vcenter{}$}%
	\sbox2{$#1\cdot\m@th$}%
	\hbox to \bigcdot@widthfactor\wd2{%
		\hfil
		\raise\ht0\hbox{%
			\scalebox{\bigcdot@scalefactor}{%
				\lower\ht0\hbox{$#1\bullet\m@th$}%
			}%
		}%
		\hfil
	}%
}
\makeatother
\newcommand{\argdot}{\,\bigcdot\,}
\newcommand{\restr}[2]{{\left.\kern-\nulldelimiterspace #1 \vphantom{\big|} \right|_{#2}}}

\newcommand{\utr}[1]{\bm{{\triangledown}#1}}
\newcommand{\ltr}[1]{\bm{{\vartriangle}#1}}

\makeatletter
\newcommand*{\transp}{%
	{\mathpalette\@transpose{}}%
}
\newcommand*{\@transpose}[2]{%
	\raisebox{\depth}{$\m@th#1\intercal$}%
}
\makeatother
\newcommand{\asconv}{\xrightarrow{\mathmakebox[0.7cm]{\mathit{a.s.}}}}
\newcommand{\wconv}{\xrightarrow{\mathmakebox[0.7cm]{w}}}
\newcommand{\defeq}{\coloneqq}
\newcommand*{\defiff}{\;\vcentcolon\Longleftrightarrow\;}
\newcommand{\idn}[1]{\bbrac{#1}}
\newcommand{\shifteq}{\sim_{\oplus}}
\newcommand{\bigO}{\mathcal{O}}

\theoremstyle{plain}
\newtheorem{sectioncount}{xxxxxxx}[section]
\newtheorem{theorem}[sectioncount]{Theorem}

\newtheorem{lemma}[sectioncount]{Lemma}
\newtheorem{corollary}[sectioncount]{Corollary}
\theoremstyle{definition}
\newtheorem{example}[sectioncount]{Example}
\newtheorem{remark}[sectioncount]{Remark}
\newtheorem{definition}[sectioncount]{Definition}

\newtheorem{propertyXXX}{Proposition}
\newenvironment{property}[1]
{\renewcommand{\thepropertyXXX}{\textup{\textbf{(#1)}}}\begin{propertyXXX}}
{\end{propertyXXX}}

\usepackage[backend = biber, maxbibnames = 10, sortcites = true, giveninits = true]{biblatex}
\usepackage{csquotes}
\DeclareListInputHandler{location}{}

\newcommand{\footremember}[2]{
	\footnote{#2}
	\newcounter{#1}
	\setcounter{#1}{\value{footnote}}
}
\newcommand{\footrecall}[1]{
	\footnotemark[\value{#1}]
}

\begin{document}

\title{Identifiability and Exact Reconstruction \\ of the Optimal Transport Cost on Finite Spaces}
\author{
\begin{tabular}{ c c c }
	Alberto Gonz\'{a}lez-Sanz\hspace*{-0.3em}\footremember{col}{Department of Statistics, Columbia University, New York, United States} &
	Michel Groppe\hspace*{-0.3em}\footremember{ims}{Institute for Mathematical Stochastics, University of G{\"o}ttingen, 37077 G{\"o}ttingen, Germany} &
	Axel Munk\hspace*{-0.3em}\footrecall{ims}\hspace*{-0.3em}\footremember{mbexc}{Cluster of Excellence ``Multiscale Bioimaging: from Molecular Machines to Networks of Excitable Cells'' (MBExC), University Medical Center, Robert-Koch-Stra{\ss}e 40, 37075 G{\"o}ttingen, Germany} \\
	{\footnotesize \href{mailto:ag4855@columbia.edu}{ag4855@columbia.edu}} &
	{\footnotesize \href{mailto:michel.groppe@uni-goettingen.de}{michel.groppe@uni-goettingen.de}} &
	{\footnotesize \href{mailto:munk@math.uni-goettingen.de}{munk@math.uni-goettingen.de}}
\end{tabular}
}

\maketitle

\begin{abstract}
	The goal of optimal transport (OT) is to find optimal assignments or matchings between data sets which minimize the total cost for a given cost function. However, sometimes the cost function is unknown but we have access to (parts of) the solution to the OT problem, e.g.\ the OT plan or the value of the objective function. Recovering the cost from such information is called inverse OT and has become recently of certain interest triggered by novel applications, e.g.\ in social science and economics. This raises the issue under which circumstances such cost is identifiable, i.e., it can be uniquely recovered from other OT quantities. In this work we provide sufficient and necessary conditions for the identifiability of the cost function on finite ground spaces. We find that such conditions correspond to the combinatorial structure of the corresponding linear program and discuss its computational complexity and implications for cost estimation in statistical linear models.
\end{abstract}



\section{Introduction}

Numerous natural or social phenomena are often described through the minimization (or maximization) of a cost (or utility) function. In economics, for instance, consumers tend to select products that maximize their utility function (see \cite{GalichonBook}). The utility of a product might depend on ease of access, e.g.\ distance to the shop it is sold at, and its price and quantity there.
Besides, minimization of the total cost, e.g.\ encoding the distance on a graph or network, can provide understanding of the organization of biological or physical systems, see e.g.\ \cite{Schiebinger2019,Bunne2023,Bunne2024,Naas2024,Cang2023,Qu2024}. In fact, recently, the applications of optimal transport theory have significantly expanded, driven by the development of faster computational algorithms that enable the handling of larger datasets. This advancement, in turn, has facilitated its broader use in data science including statistics and machine learning, see e.g.\ \cite{Peyre2018,Kuhn2019,Montesuma2024,Panaretos2020,Sommerfeld2018} and the references given there.

Often these problems can be conveniently mathematically formulated in the optimal transport (OT) framework: Move mass from the target distribution $\mu$ to the source distribution $\nu$ (which are assumed here to have the same total mass) with the least amount of total effort according to some cost function $c$. In this paper, we focus on the situation that the distributions $\mu$ and $\nu$ are supported on a finite ground space, respectively, i.e., $\mu$ and $\nu$ can be regarded as probability vectors. To be more precise, for $N \in \N$ denote with $\Delta_N \defeq \{ x \in \R^N_+ : \sum_{i=1}^N x_i = 1 \}$ the $N$-dimensional probability simplex, where, $\R_+$ stands for the set of nonnegative real numbers. Given a cost matrix $c \in \R^{N \times M}$, the OT problem between two probability vectors $\mu \in \Delta_N$, $\nu \in \Delta_M$ is then to solve the minimization problem
\begin{equation} \label{eq:OT}
	\OT_c(\mu,\nu) \defeq \min_{\pi \in \Pi(\mu,\nu)} \inner{c}{\pi}\,,
\end{equation}
where $\inner{\argdot}{\argdot}$ is the Frobenius inner product and $\Pi(\mu,\nu)$ represents the set of probability matrices of size $N \times M$ with marginals $\mu$ and $\nu$, i.e.,
\begin{equation*}
    \Pi(\mu, \nu) \defeq \biggl\{  \pi \in \R^{N \times M}_+ : \sum_{s=1}^M \pi_{i,s} = \mu_i,\, \sum_{r = 1}^N \pi_{r,j} = \nu_j \text{ for all } i \in \idn{N},\, j \in \idn{M} \biggr\},
\end{equation*}
for $\idn{N} \defeq \{1, \ldots, N\}$. We call any minimizer $\pi$ of \eqref{eq:OT} an \textit{OT plan} between $\mu$ and $\nu$. The value $\OT_c(\mu,\nu)$ in \eqref{eq:OT} gives the \textit{total cost} of the OT problem. Furthermore, we have the dual formulation (cf.\ \cite{Luenberger2008})
\begin{equation} \label{eq:OT_dual}
	\OT_c(\mu, \nu) = \max_{(f,g) \in \Phi_c} \inner{f}{\mu} +  \inner{g}{\nu}\,,
\end{equation}
where $\Phi_c \defeq \{ (f, g) \in \R^N \times \R^M \mid f_{i} + g_{j} \leq c_{i,j} \text{ for all } i \in \idn{N},\,j \in \idn{M}\}$. The elements  of a maximizing pair of \eqref{eq:OT_dual} are called \textit{optimal potentials}.

In specific settings, the OT plan can represent an optimal assignment w.r.t.\ the cost $c$ given by an $N \times N$ matrix (here $N=M$) with only one nonzero entry in each row or column. In the discussed economic models, it could represent the optimal assignment of workers to firms or buyers to sellers. The theory of optimal transport on finite ground spaces is classical since the pioneering work of Kantorovitch \cite{Kantorovitch1942,Kantorovitch1958} and can be cast in the framework of linear programming \cite{Bertsimas1997,Luenberger2008}. See also \cite{Brezis2018} for an elementary proof of \eqref{eq:OT_dual}, and the references given there. Once the cost function is specified, the OT plan and total cost can be computed (numerically) from the marginals $\mu$ and $\nu$, see \cite{Peyre2018} for various methods.

However, in many instances, the cost function is (partially) unknown and must be inferred from the observed optimal matching. This is known as inverse optimal transport and was introduced by \cite{InverseSIam} to predict the flow of migrants using country-specific characteristics and pairwise measures
of dissimilarity between countries. Since then, inverse OT has received great attention, see e.g.\ \cite{Dupuy2016EstimatingMA,InverseJMLR2023,InversePMLR,InverseSIam,InverseCarlier,andrade2023sparsistency}, where further methods in more complex scenarios for estimation of the cost matrix are provided. We note that all the works cited above are concerned with approximate reconstruction of the cost $c$, typically involving entropic regularization. In this work, we focus on exact reconstruction.

Furthermore, going further back in the literature, as noted in \cite{Dupuy2016EstimatingMA}, related problems have already been studied in sociology and economics. In particular, in the so-called marriage market, one seeks to infer from the observed characteristics of married couples which attributes are most relevant for matching \cite{Becker}. Another interesting scenario has been brought to our attention by one of the referees: Suppose that we observe the OT plan, as well as \textit{one} of the optimal potentials. This could correspond to observing which buyer purchases which house, and at what price. Inferring the cost function from this would gain insight into the house preferences of buyers. More recent contributions include applications to contrastive learning \cite{InverseJMLR2023}, the characterization of the manifold of cross-ratio-equivalent costs together with an analysis of the implications of model priors and the derivation of an MCMC sampler \cite{InversePMLR}, migration modeling \cite{InverseCarlier,InverseSIam}, and global trade modeling \cite{gaskin.duncan.2025modellingglobaltradeoptimal}.

However, to put such methods on solid mathematical grounds, it is key to understand when the observed data uniquely determines the cost matrix, which is the goal of this work. That is, we describe conditions on the data so that  the cost matrix is identifiable, a terminology we borrowed from the theory of statistical models, where it has a long history (see e.g.\ \cite{Teicher1963,Lindsay1995,Koopmans1950}). We focus on finitely supported probability measures, which complements our previous work \cite{gonzalezsanz2024nonlinearinverseoptimaltransport} in the continuous case. Throughout most parts of the paper we will assume that these data are observed exactly without measurement noise, however, extensions to statistical models will be discussed in \autoref{Sec:Estimation}.

Our observations can be described as follows. We observe $K \in \N$ pairs of marginals $(\mu^{(k)},\nu^{(k)}) \in \Delta_N \times \Delta_M$, $k \in \idn{K}$, and (a subset of) the OT information w.r.t.\ the underlying cost matrix $c$:

\parbox[c]{0.8\linewidth}{
\begin{enumerate}[(i)]
    \item\label{enum:ot_cost} The corresponding total costs $\alpha^{(k)} \in \R$, $k \in \idn{K}$, i.e., the optimal values $\alpha^{(k)} = \OT_c(\mu^{(k)}, \nu^{(k)}) $ in \eqref{eq:OT},
    \item\label{enum:ot_plan} OT plans $\pi^{(k)} \in \Pi(\mu^{(k)}, \nu^{(k)})$, $k \in \idn{K}$, and
    \item\label{enum:ot_pot} optimal potentials $(f^{(k)}, g^{(k)}) \in \Phi_c$, $k \in \idn{K}$, such that
\end{enumerate}} \hfill \hypertarget{eq:obs}{(OBS)}

\begin{equation*}
	\alpha^{(k)} = \inner{c}{\pi^{(k)}} = \inner{f^{(k)}}{\mu^{(k)}} +  \inner{g^{(k)}}{\nu^{(k)}}\,.
\end{equation*}
Our aim is to recover the cost matrix $c$ from (a combination of) the quantities in \ref{enum:ot_cost}--\ref{enum:ot_pot}, which are either observed directly or are corrupted by noise (see \autoref{Sec:Estimation}). We call $c$ \textit{identifiable} if it is uniquely determined by (parts of) \hyperlink{eq:obs}{(OBS)}, i.e., if (a subset of) the system $\{\mu^{(k)}, \nu^{(k)}, \alpha^{(k)}, \pi^{(k)}, f^{(k)}, g^{(k)}\}_{k=1}^K$ uniquely determines $c$. Furthermore, we say that $c$ is \textit{consistent} with the observed OT information if \hyperlink{eq:obs}{(OBS)} holds. Note that, when considering identifiability, we assume that the underlying cost matrix $c$ exists, i.e., there is at least one cost matrix that is consistent with the observed OT information.

To avoid confusion of terminology, it is instructive to distinguish our notion of ``inverse optimal transport'' from inverse linear programming. Recall that the forward OT problem, i.e., the computation of the total cost $\OT_c(\mu, \nu)$ and a corresponding OT plan $\pi$ in \eqref{eq:OT} on finite spaces, amounts to solving a linear program (LP), i.e., for $A \in \R^{E \times D}$, $\vc{c} \in \R^{D}$ and $\vc{b} \in \R^{E}$ the minimization problem
\begin{equation*}
     \min_{\vc{x} \in P(\vc{b})} \inner{\vc{c}}{  \vc{x} }  \qquad \text{ where } \quad P(\vc{b}) \defeq \{ \vc{x} \in \R^{D} : A \vc{x} = \vc{b}\,, \vc{x} \geq 0 \} \,.
\end{equation*}
In contrast, given an observation $\vc{x}' \in \R^{D}$ and prior information $\vc{c}' \in \R^{D}$, the goal of inverse LP is to solve the following minimization problem
\begin{equation} \label{eq:inverse_lp}
    \min_{\vc{c}} \{ \norm{\vc{c} - \vc{c}'} : \vc{x}' \in  \argmin_{\vc{z} \in P(\vc{b})} \inner{\vc{c}}{\vc{z}} \}\,,
\end{equation}
where $\norm{\argdot}$ is some (semi-)norm on $\R^{D}$. Hence, inverse LP tries to find a $\vc{c}$ close to $\vc{c}'$ such that $\vc{x}'$ is an optimal solution to the corresponding LP whereas inverse OT aims to recover $c$ from (a combination of) quantities in \ref{enum:ot_cost}--\ref{enum:ot_pot}. In particular, when the LP is the OT problem, then \eqref{eq:inverse_lp} corresponds to the inverse OT problem in the case that only one OT plan \ref{enum:ot_plan} is observed, and the objective is assumed to be constant (instead of $c \mapsto \norm{c - c'}$). The inverse LP problem has been intensively studied in the optimization literature, see e.g.\ \cite{DempeInverLP,Chan2023,Ahuja2001,Iyengar2005,JianzhongZhang1996,Zhang1999,Tayyebi2018} and references therein. As for inverse OT, most of the work focuses on algorithmic methods to obtain a solution $\vc{c}$ to \eqref{eq:inverse_lp}. To put inverse LP into the perspective of the present work, note that identifiability of a solution to the inverse LP problem \eqref{eq:inverse_lp} does not hold in general and has not been considered so far. Similarly, this is the case for noisy inverse LP, see e.g.\ \cite{Aswani2018,Chan2019,Chan2023}. Furthermore, inverse LP often assumes the observation of a single solution of the LP. Our problem considered is different, however, it can be cast in the framework of inverse problems where a forward operator has to be inverted. To the best of our knowledge, the only result relevant for identifiability in the sense of the present paper is given in \cite{Tavaslioglu2018} where they characterize the set of consistent cost matrices in the case where only one (meaning $K=1$) OT plan \ref{enum:ot_plan} is observed. This, however, does not generalize to $K > 1$ and is thus of limited usability for identifiability, see \autoref{rem:inv_feas_reg} for further discussion.

To be concise, we present our results in the OT framework, but they can also be formulated for more general linear programs. We discuss this in \autoref{rem:lin_prog}.

\subsection{Summary of main results}
The main finding of this paper is that identifiability is governed by the polyhedral geometry of the associated transport polytopes, and more precisely by the structure of their optimal faces. The situation is quite complex and therefore we will address this in five separate cases.

\paragraph{Scenario A} When only the total costs \ref{enum:ot_cost} are observed, identifiability is characterized by the unique solvability of the family of linear programs associated with all admissible choices of optimal extreme points (cf.~\autoref{thm:ident_only_total}).

\paragraph{Scenario B} When optimal potentials \ref{enum:ot_pot} are additionally observed\footnote{As knowing the optimal potentials (together with the marginals) implies knowledge of the total costs, scenario B can be equivalently stated as only observing the optimal potentials.}, the primal-dual optimality relations impose further linear restrictions and yield a sharper necessary and sufficient criterion (see \autoref{thm:ident_only_pot}).
\begin{figure}
    \centering

\tikzset{every picture/.style={line width=0.5pt}} 

\begin{tikzpicture}[x=0.5pt,y=0.5pt,yscale=-1,xscale=1]

\draw  [fill={rgb, 255:red, 184; green, 233; blue, 134 }  ,fill opacity=0.42 ] (153.75,157.92) -- (70.56,143.38) -- (52.46,99.06) -- (124.46,86.21) -- (187.06,122.59) -- cycle ;
\draw    (187.06,122.59) -- (208.06,194.59) ;
\draw    (153.75,157.92) -- (157,233) ;
\draw    (70.56,143.38) -- (60,232) ;
\draw    (52.46,99.06) -- (24,190) ;
\draw   (369.75,157.92) -- (286.56,143.38) -- (268.46,99.06) -- (340.46,86.21) -- (403.06,122.59) -- cycle ;
\draw    (403.06,122.59) -- (424.06,194.59) ;
\draw    (369.75,157.92) -- (373,233) ;
\draw    (286.56,143.38) -- (276,232) ;
\draw    (268.46,99.06) -- (240,190) ;
\draw   (586.75,159.92) -- (503.56,145.38) -- (485.46,101.06) -- (557.46,88.21) -- (620.06,124.59) -- cycle ;
\draw    (620.06,124.59) -- (641.06,196.59) ;
\draw    (586.75,159.92) -- (590,235) ;
\draw    (503.56,145.38) -- (493,234) ;
\draw    (485.46,101.06) -- (457,192) ;
\draw [color={rgb, 255:red, 126; green, 211; blue, 33 }  ,draw opacity=1 ][line width=1.5]    (286.56,143.38) -- (369.75,157.92) ;
\draw  [color={rgb, 255:red, 65; green, 117; blue, 5 }  ,draw opacity=1 ][fill={rgb, 255:red, 65; green, 117; blue, 5 }  ,fill opacity=1 ] (117.66,121.83) .. controls (117.66,119.25) and (119.75,117.16) .. (122.33,117.16) .. controls (124.91,117.16) and (127,119.25) .. (127,121.83) .. controls (127,124.41) and (124.91,126.51) .. (122.33,126.51) .. controls (119.75,126.51) and (117.66,124.41) .. (117.66,121.83) -- cycle ;
\draw  [color={rgb, 255:red, 65; green, 117; blue, 5 }  ,draw opacity=1 ][fill={rgb, 255:red, 65; green, 117; blue, 5 }  ,fill opacity=1 ] (328.15,150.65) .. controls (328.15,148.07) and (330.24,145.98) .. (332.83,145.98) .. controls (335.41,145.98) and (337.5,148.07) .. (337.5,150.65) .. controls (337.5,153.23) and (335.41,155.33) .. (332.83,155.33) .. controls (330.24,155.33) and (328.15,153.23) .. (328.15,150.65) -- cycle ;
\draw  [color={rgb, 255:red, 65; green, 117; blue, 5 }  ,draw opacity=1 ][fill={rgb, 255:red, 65; green, 117; blue, 5 }  ,fill opacity=1 ] (582.75,159.92) .. controls (582.75,157.34) and (584.84,155.25) .. (587.42,155.25) .. controls (590,155.25) and (592.09,157.34) .. (592.09,159.92) .. controls (592.09,162.5) and (590,164.6) .. (587.42,164.6) .. controls (584.84,164.6) and (582.75,162.5) .. (582.75,159.92) -- cycle ;

\draw (80,211) node [anchor=north west][inner sep=0.75pt]   [align=left] {$\displaystyle \Pi ( \mu ,\nu )$};
\draw (298,212) node [anchor=north west][inner sep=0.75pt]   [align=left] {$\displaystyle \Pi ( \mu ,\nu )$};
\draw (517,210) node [anchor=north west][inner sep=0.75pt]   [align=left] {$\displaystyle \Pi ( \mu ,\nu )$};
\draw (101,101) node [anchor=north west][inner sep=0.75pt]   [align=left] {$\displaystyle \pi $};
\draw (326,123) node [anchor=north west][inner sep=0.75pt]   [align=left] {$\displaystyle \pi $};
\draw (577,134) node [anchor=north west][inner sep=0.75pt]   [align=left] {$\displaystyle \pi $};

\end{tikzpicture}
    \caption{Illustration of the information obtained depending on the position of the observed plan. In all cases,  the smallest face (in light green) that contains the optimal plan $\pi$ (in dark green) is contained in the set of solutions. However, the situation displayed in the figure on the left provides more information than the face on the right.}
    \label{fig:Figure:polytope-intro}
\end{figure}

\paragraph{Scenario C} In the regime where only the optimal plans \ref{enum:ot_plan} are observed, the information content of a given observation depends on the position of the observed plan within the optimality set. \autoref{pr:ri} implies that one can recover the smallest face of the transport polytope containing the observed plan, and that this face is contained in the corresponding optimal face. Consequently, if the observed plan lies in the relative interior of the optimal face, then this minimal face coincides with the whole optimal face, so the latter is recovered exactly from the observation (see \autoref{fig:Figure:polytope-intro}). This is therefore the most informative configuration. By contrast, if the observed plan is itself an extreme point---as is the case, for instance, for outputs of the simplex algorithm---then one only recovers a single optimal vertex, which is the least informative situation. Because of shift invariance \eqref{eq:shift_equiv}, observing plans alone yields identifiability only in the quotient space modulo additive shifts, up to a residual linear subspace whose dimension is determined by the span generated by the observed plans and the extreme points of the recovered faces (see \autoref{thm:ident_only_plans}).

\paragraph{Scenario D} Once total costs \ref{enum:ot_cost} are added, this obstruction disappears and  \autoref{thm:ident_total_plan_eq} shows that identifiability reduces to the unique solvability of a linear system.

\paragraph{Scenario E} If we only observe the OT plans \ref{enum:ot_plan} and one of the optimal potentials \ref{enum:ot_pot}, we are again in a scenario where the cost matrix is only identifiable up to certain additive shifts. In this case, we find that we have identifiability if the supports of the OT plans are sufficiently overlapping (\autoref{thm:ident_plans_one_pot}).

\paragraph{Scenario F} Under full information \ref{enum:ot_cost}--\ref{enum:ot_pot}, we obtain a complete characterization: the cost matrix is identifiable if and only if every matrix entry is covered by the support of at least one observed optimal plan (cf.~\autoref{thm:ident_plan_pot}).

Further, in each of the above scenarios we also discuss the computational complexity of the identifiability criteria.

In some cases, we may have prior information about the cost function's structure. An important example is the symmetry of the cost matrix. This will be addressed in \autoref{Sec:struct}, where we give identifiability conditions under the observation of total costs and OT plans (see \autoref{thm:ident_total_plan_eq_sym}) and under the exclusive observation of OT plans (see \autoref{thm:ident_only_plans_sym}).

As an application of our theory, in \autoref{Sec:Estimation} we propose an estimator for the cost function and derive consistency results and asymptotic normality in a statistical setting, as well as asymptotic confidence sets for the cost matrix, see \autoref{thm:CLT}. Finally, \autoref{thm:BoundsLasso} is specifically tailored to sparse costs and we give recovery guarantees in this case. To this end we introduce a modified notion of sparsity that is adapted to OT and that enjoys corresponding results to statistical sparse recovery.

\begin{table}
	\centering
	\begin{tabular}{c | c c c c c c c c}
	  Subsection & \ref{Sec:TotalCost} & \ref{Sec:OTpotent} & \ref{Sec:OTpotent} & \ref{Sec:Plans} & \ref{Sec:planCost} & \ref{Sec:plans_one_pot} & \ref{Sec:FullInf} & \ref{Sec:FullInf} \\
	  Scenario & A & B & B  &  C & D & E & F & F \\\hline \rule{0pt}{2.6ex}
	 Total costs \ref{enum:ot_cost}       & $\times$ &          & $\times$ &          & $\times$ &           &          & $\times$ \\
	 OT plans \ref{enum:ot_plan}          &          &          &          & $\times$ & $\times$ & $\times$  & $\times$ & $\times$ \\
	 Optimal potentials \ref{enum:ot_pot} &          & $\times$ & $\times$ &          &          & $*$          & $\times$ & $\times$
\end{tabular}
	\caption{Subsections for the different scenarios of observing different OT related quantities. A ``$\times$'' indicates that the corresponding quantities are observed in the scenario and ``$*$'' means that we only observe one of the optimal potentials.} \label{tab:sections}
\end{table}

\subsection{Organization of the Work}\label{sec:Organization}

The rest of the paper is organized as follows: In \autoref{sec:Main}, we present our main results. This section is divided into subsections, each of which analyzes the identifiability of the cost function under scenarios A--F of observing different OT related quantities in \ref{enum:ot_cost}--\ref{enum:ot_pot}. They are organized such that gradually information is added such that identifiability becomes easier. In \autoref{sec:Notation}, we introduce the notation used throughout the paper, as well as some general results that will be utilized later. The main result of \autoref{Sec:TotalCost} is \autoref{thm:ident_only_total}, which provides necessary and sufficient conditions for identifiability when only the total costs \ref{enum:ot_cost} are observed (Scenario A). In \autoref{Sec:OTpotent}, we derive the necessary and sufficient conditions for identifiability when only the potentials \ref{enum:ot_pot} are observed (Scenario B, see \autoref{thm:ident_only_pot}). Then, \autoref{thm:ident_only_plans} in \autoref{Sec:Plans} states the identifiability conditions in the setting of only observing the transport plans \ref{enum:ot_plan} (Scenario C). \autoref{Sec:planCost} states necessary and sufficient conditions under the observation of both total costs \ref{enum:ot_cost} and transport plans \ref{enum:ot_plan} (Scenario D, see \autoref{thm:ident_total_plan}). In \autoref{Sec:plans_one_pot} we give a sufficient condition when only the OT plans \ref{enum:ot_plan} and one of the optimal potentials \ref{enum:ot_pot} are observed (Scenario E, \autoref{thm:ident_plans_one_pot}). Finally, the case of full information \ref{enum:ot_cost}--\ref{enum:ot_pot} (Scenario F) is covered by \autoref{thm:ident_plan_pot} in \autoref{Sec:FullInf}, where necessary and sufficient conditions are given. To ease navigation, \autoref{tab:sections} summarizes which subsection deals with which scenario. In \autoref{Sec:struct} we deal with symmetric cost matrices and \autoref{Sec:Estimation} is devoted to statistical applications.

\section{Main Results. Identifiability of the Cost Matrix}\label{sec:Main}

This section contains the main results of this paper. In each subsection we derive necessary and sufficient conditions for the identifiability of the cost matrix $c$ considering the different possible combinations of the available information. We refer to \autoref{sec:Organization} for the organization of the different subsections. Note that in all the cases the marginals are assumed to be known (e.g.\ implicitly given by the knowledge of OT plans).

\subsection{Notation and Preliminary Results}\label{sec:Notation}

As the OT problem \eqref{eq:OT} is an instance of a linear program (LP), we can make use of the geometry of LPs to achieve identifiability of the cost matrix $c$. To this end, we first introduce the needed tools and terminology from linear programming \cite{Bertsimas1997,Luenberger2008}.

\begin{figure}
\centering

\tikzset{every picture/.style={line width=0.75pt}} 

\begin{tikzpicture}[x=0.75pt,y=0.75pt,yscale=-1,xscale=1]

\draw  [fill={rgb, 255:red, 219; green, 219; blue, 219 }  ,fill opacity=1 ] (424.92,81.76) -- (478.3,93.3) -- (461.3,128.3) -- (402.3,122.3) -- (378.3,93.3) -- cycle ;
\draw [fill={rgb, 255:red, 219; green, 219; blue, 219 }  ,fill opacity=1 ]   (378,131.8) -- (378.3,93.3) -- (402.3,122.3) -- (402.6,156) ;
\draw [fill={rgb, 255:red, 219; green, 219; blue, 219 }  ,fill opacity=1 ]   (402.6,156) -- (402.3,122.3) -- (461.3,128.3) -- (468.6,153.4) ;
\draw [fill={rgb, 255:red, 219; green, 219; blue, 219 }  ,fill opacity=1 ]   (468.6,153.4) -- (461.3,128.3) -- (478.3,93.3) -- (477,134.6) ;
\draw  [fill={rgb, 255:red, 0; green, 0; blue, 0 }  ,fill opacity=1 ] (376.89,93.3) .. controls (376.89,92.52) and (377.52,91.89) .. (378.3,91.89) .. controls (379.08,91.89) and (379.71,92.52) .. (379.71,93.3) .. controls (379.71,94.08) and (379.08,94.71) .. (378.3,94.71) .. controls (377.52,94.71) and (376.89,94.08) .. (376.89,93.3) -- cycle ;
\draw  [fill={rgb, 255:red, 0; green, 0; blue, 0 }  ,fill opacity=1 ] (423.51,81.76) .. controls (423.51,80.98) and (424.14,80.35) .. (424.92,80.35) .. controls (425.69,80.35) and (426.32,80.98) .. (426.32,81.76) .. controls (426.32,82.54) and (425.69,83.17) .. (424.92,83.17) .. controls (424.14,83.17) and (423.51,82.54) .. (423.51,81.76) -- cycle ;
\draw  [fill={rgb, 255:red, 0; green, 0; blue, 0 }  ,fill opacity=1 ] (476.89,93.3) .. controls (476.89,92.52) and (477.52,91.89) .. (478.3,91.89) .. controls (479.08,91.89) and (479.71,92.52) .. (479.71,93.3) .. controls (479.71,94.08) and (479.08,94.71) .. (478.3,94.71) .. controls (477.52,94.71) and (476.89,94.08) .. (476.89,93.3) -- cycle ;
\draw  [fill={rgb, 255:red, 0; green, 0; blue, 0 }  ,fill opacity=1 ] (400.89,122.3) .. controls (400.89,121.52) and (401.52,120.89) .. (402.3,120.89) .. controls (403.08,120.89) and (403.71,121.52) .. (403.71,122.3) .. controls (403.71,123.08) and (403.08,123.71) .. (402.3,123.71) .. controls (401.52,123.71) and (400.89,123.08) .. (400.89,122.3) -- cycle ;
\draw  [fill={rgb, 255:red, 0; green, 0; blue, 0 }  ,fill opacity=1 ] (459.89,128.3) .. controls (459.89,127.52) and (460.52,126.89) .. (461.3,126.89) .. controls (462.08,126.89) and (462.71,127.52) .. (462.71,128.3) .. controls (462.71,129.08) and (462.08,129.71) .. (461.3,129.71) .. controls (460.52,129.71) and (459.89,129.08) .. (459.89,128.3) -- cycle ;
\draw  [fill={rgb, 255:red, 219; green, 219; blue, 219 }  ,fill opacity=1 ] (281.92,95.36) -- (335.3,106.9) -- (318.3,141.9) -- (259.3,135.9) -- (235.3,106.9) -- cycle ;
\draw  [fill={rgb, 255:red, 0; green, 0; blue, 0 }  ,fill opacity=1 ] (233.89,106.9) .. controls (233.89,106.12) and (234.52,105.49) .. (235.3,105.49) .. controls (236.08,105.49) and (236.71,106.12) .. (236.71,106.9) .. controls (236.71,107.68) and (236.08,108.31) .. (235.3,108.31) .. controls (234.52,108.31) and (233.89,107.68) .. (233.89,106.9) -- cycle ;
\draw  [fill={rgb, 255:red, 0; green, 0; blue, 0 }  ,fill opacity=1 ] (280.51,95.36) .. controls (280.51,94.58) and (281.14,93.95) .. (281.92,93.95) .. controls (282.69,93.95) and (283.32,94.58) .. (283.32,95.36) .. controls (283.32,96.14) and (282.69,96.77) .. (281.92,96.77) .. controls (281.14,96.77) and (280.51,96.14) .. (280.51,95.36) -- cycle ;
\draw  [fill={rgb, 255:red, 0; green, 0; blue, 0 }  ,fill opacity=1 ] (333.89,106.9) .. controls (333.89,106.12) and (334.52,105.49) .. (335.3,105.49) .. controls (336.08,105.49) and (336.71,106.12) .. (336.71,106.9) .. controls (336.71,107.68) and (336.08,108.31) .. (335.3,108.31) .. controls (334.52,108.31) and (333.89,107.68) .. (333.89,106.9) -- cycle ;
\draw  [fill={rgb, 255:red, 0; green, 0; blue, 0 }  ,fill opacity=1 ] (257.89,135.9) .. controls (257.89,135.12) and (258.52,134.49) .. (259.3,134.49) .. controls (260.08,134.49) and (260.71,135.12) .. (260.71,135.9) .. controls (260.71,136.68) and (260.08,137.31) .. (259.3,137.31) .. controls (258.52,137.31) and (257.89,136.68) .. (257.89,135.9) -- cycle ;
\draw  [fill={rgb, 255:red, 0; green, 0; blue, 0 }  ,fill opacity=1 ] (316.89,141.9) .. controls (316.89,141.12) and (317.52,140.49) .. (318.3,140.49) .. controls (319.08,140.49) and (319.71,141.12) .. (319.71,141.9) .. controls (319.71,142.68) and (319.08,143.31) .. (318.3,143.31) .. controls (317.52,143.31) and (316.89,142.68) .. (316.89,141.9) -- cycle ;
\draw  [fill={rgb, 255:red, 0; green, 0; blue, 0 }  ,fill opacity=1 ] (135.69,126.3) .. controls (135.69,125.52) and (136.32,124.89) .. (137.1,124.89) .. controls (137.88,124.89) and (138.51,125.52) .. (138.51,126.3) .. controls (138.51,127.08) and (137.88,127.71) .. (137.1,127.71) .. controls (136.32,127.71) and (135.69,127.08) .. (135.69,126.3) -- cycle ;
\draw  [fill={rgb, 255:red, 0; green, 0; blue, 0 }  ,fill opacity=1 ] (182.31,114.76) .. controls (182.31,113.98) and (182.94,113.35) .. (183.72,113.35) .. controls (184.49,113.35) and (185.12,113.98) .. (185.12,114.76) .. controls (185.12,115.54) and (184.49,116.17) .. (183.72,116.17) .. controls (182.94,116.17) and (182.31,115.54) .. (182.31,114.76) -- cycle ;
\draw    (138.51,126.3) -- (183.72,114.76) ;
\draw  [fill={rgb, 255:red, 0; green, 0; blue, 0 }  ,fill opacity=1 ] (88.89,120.1) .. controls (88.89,119.32) and (89.52,118.69) .. (90.3,118.69) .. controls (91.08,118.69) and (91.71,119.32) .. (91.71,120.1) .. controls (91.71,120.88) and (91.08,121.51) .. (90.3,121.51) .. controls (89.52,121.51) and (88.89,120.88) .. (88.89,120.1) -- cycle ;

\end{tikzpicture}

\caption{Four different polytopes (from left to right): A zero-dimensional polytope (point), one-dimensional polytope (line segment), two-dimensional polytope (polygon) and part of a three-dimensional polytope (polyhedron). Note that each depicted polytope is also a face of the higher-dimensional ones and the convex hull of its extreme points.} \label{fig:polytope}

\end{figure}

Let $(\mu, \nu) \in \Delta_N \times \Delta_M$ be a pair of marginals and $c \in \R^{N \times M}$ a cost matrix. We denote the by column-major order vectorized version of $c \in \R^{N \times M}$ by
\begin{equation*}
	\vc{c} \defeq (\vc{c}_1, \dots, \vc{c}_{NM})^\transp
	\defeq (c_{1,1}, \dots, c_{N,1}, \dots,  c_{1,M}, \dots, c_{N,M})^\transp \in \R^{N M}\,.
\end{equation*} The set of transport plans $\Pi(\mu, \nu)$ is a (bounded) polytope. Furthermore, denote the non-empty and compact set of OT plans with
\begin{equation*}
    \OPT_c(\mu, \nu )\defeq \argmin_{\pi \in \Pi(\mu,\nu)} \inner{c}{\pi}\,.
\end{equation*}
A subset $F \subseteq \Pi(\mu, \nu)$ is called a face if there exists a cost matrix $c$ such that $F = \OPT_c(\mu, \nu)$. An element $\pi \in \Pi(\mu, \nu)$ is called an extreme point (or vertex) if $\{ \pi \}$ is a face or equivalently, if $\pi$ is not a convex combination of two different points in $\Pi(\mu, \nu)$. Denote the set of extreme points of $\Pi(\mu, \nu)$ with $\EP(\mu, \nu)$. As a shortcut, we may write for $\pi \in \Pi(\mu, \nu)$ that $\Pi(\pi) \defeq \Pi(\mu, \nu)$ and $\EP(\pi) \defeq \EP(\mu, \nu)$ as well as $\OPT_c(\pi) = \OPT_c(\mu, \nu)$. It follows that every face $F$ of $\Pi(\mu, \nu)$, including $\Pi(\mu, \nu)$ itself, is the convex hull of a unique set of extreme points, i.e., there exist $u^{(1)}, \ldots, u^{(L_F)} \in \EP(\mu, \nu)$ such that
\begin{equation*}
    F = \coh \{ u^{(\ell)} : \ell \in \idn{L_F} \} \defeq \left\{ \sum_{\ell=1}^{L_F} \lambda_\ell\, u^{(\ell)} : \lambda \in \Delta_{L_F} \right\} \,.
\end{equation*}
Here, $(L_F-1)$ is the dimension of the face $F$. See \autoref{fig:polytope} for examples of polytopes and faces. In particular, when observing the OT plans \ref{enum:ot_plan} we can find the faces with minimal dimension that contain them. These faces must then be contained in the set of optimizers, see \autoref{fig:Figure:polytope-intro}. We summarize this observation for later use in the following:

\begin{property}{RI} \label{pr:ri}
    Suppose that we observe $K$ pairs of marginals and the corresponding OT plans \ref{enum:ot_plan}. Then, for each $k \in \idn{K}$, there exist $L_k \in \N$ and $u^{(1, k)}, \ldots, u^{(L_k,k)} \in \EP(\mu^{(k)}, \nu^{(k)})$ such that
    \begin{equation*}
 	 \pi^{(k)} \in \left\{ \sum_{\ell=1}^{L_k} \lambda_\ell u^{(\ell, k)} : \lambda \in \Delta_{L_k}, \lambda > 0 \right\}
    \end{equation*}
	and $\coh \{ u^{(\ell, k)} : \ell \in \idn{L_k} \} \subseteq \OPT_c(\pi^{(k)})$.
\end{property}

\begin{remark}[Information content of \ref{pr:ri}] \label{rem:more_vertex_better}
	\autoref{pr:ri} shows that $\pi^{(k)}$ carries more information about $\OPT_c(\mu^{(k)}, \nu^{(k)})$ for bigger $L_k$. The least informative case is $L_k = 1$ when $\pi^{(k)}$ is an extreme point, and the most informative case is when $\pi^{(k)}$ lies in the relative interior of $\OPT_c(\mu^{(k)}, \nu^{(k)})$. In particular, if the OT plan $\pi^{(k)}$ is computed with the simplex algorithm, then it is an extreme point and \autoref{pr:ri} is least informative.
\end{remark}

\begin{remark}[Computation of \ref{pr:ri}] \label{rem:comp_ri}
    Note that \autoref{pr:ri} requires explicit knowledge of the extreme points $\EP(\mu, \nu)$. Given a pair of marginals $(\mu, \nu)$, the polytope $\Pi(\mu, \nu)$ is determined via a number of linear equations. Calculating the extreme points from this is called the vertex enumeration problem. There are various algorithms available for this, see e.g.\ \cite{Avis1992,Bremner1998,Dyer1977} and references therein. Naturally, the computational complexity depends on the number of vertices $V$ of which there can be exponentially (in $N$, $M$) many. For example, the well-known Birkhoff polytope (for $N = M$) has $V = N!$ vertices. Hence, the enumeration of the vertices is, in general, computationally expensive. For instance, the Avis--Fukuda algorithm \cite{Avis1992} takes $\bigO(N^3 V)$ time (for $N = M$). However, note that if only the extreme points $\{ u^{(\ell, k)} : \ell \in \idn{L_k}\}$ in \autoref{pr:ri} are needed, we do not have to compute the whole set of extreme points $\EP(\pi^{(k)})$. Namely, it suffices to only consider the sub-polytope
	\begin{equation*}
		\tilde{\Pi}(\pi^{(k)}) \defeq \{ \pi \in \Pi(\pi^{(k)}) : \supp \pi \subseteq \supp \pi^{(k)} \}\,,
	\end{equation*}
	where $\supp \pi^{(k)}$ is defined in \eqref{eq:supp}, i.e., run the vertex enumeration algorithm on this sub-polytope (here $V = L_k$). This can decrease the computational complexity, however, the worst case $V = \abs{\EP(\pi^{(k)})}$ is still possible.
\end{remark}

As an instance of an LP, the OT problem has the primal \eqref{eq:OT} and dual formulation \eqref{eq:OT_dual}. By general LP theory, their solutions are linked in the following way \cite[Section~4.4]{Luenberger2008}:

\begin{lemma}[Primal-dual optimality criterion] \label{lemma:primal_dual_opt}
	Let $(\mu, \nu) \in \Delta_N \times \Delta_M$ be a pair of marginals and $c \in \R^{N \times M}$ a cost matrix. A transport plan $\pi \in \Pi(\mu, \nu)$ and potentials $(f, g) \in \Phi_c$ are optimal if and only if for all $(i, j) \in \idn{N} \times \idn{M}$ it holds that $\pi_{i, j} > 0$ implies $c_{i, j} = f_i + g_j$.
\end{lemma}

\autoref{lemma:primal_dual_opt}  links identifiability of the cost matrix $c$ to the support of an OT plan $\pi$, denoted as
\begin{equation} \label{eq:supp}
	\supp \pi \defeq \{ (i, j) \in \idn{N} \times \idn{M} : \pi_{i, j} > 0 \},
\end{equation}
and knowledge of the optimal potentials $f$ and $g$. Upon defining the outer sum
\begin{equation*}
	f \oplus g \defeq [f_i + g_j]_{i,j=1}^{N,M},
\end{equation*}
the condition in \autoref{lemma:primal_dual_opt} can be stated as $f \oplus g = c$ on $\supp \pi$.

 We say that two cost matrices $c,\, c' \in \R^{N \times M}$ are shift equivalent if
\begin{equation} \label{eq:shift_equiv}
	c \shifteq c' \quad \defiff \quad c - c' = a \oplus b \quad \text{for some vectors }	a \in \R^N\,, b \in \R^M\,.
\end{equation}
Note that this defines an equivalence relation on $\R^{N \times M}$. Denote with $\shiftc^{N \times M} \defeq \R^{N \times M} / { \shifteq }$ the corresponding quotient space. When we say that $c$ is identifiable in $\shiftc^{N \times M}$, we mean that $c$ is identifiable in $\R^{N \times M}$ up to shift equivalence. This definition is motivated by the fact that for any pair of marginals $(\mu, \nu) \in \Delta_N \times \Delta_M$ it holds for $c \shifteq c'$ that $\OPT_{c}(\mu, \nu) = \OPT_{c'}(\mu, \nu)$, i.e., the OT w.r.t.\ shift equivalent cost matrices has the same OT plans. Therefore, for identifiability of the cost matrix $c$ it is crucial to know the total costs \ref{enum:ot_cost}.  In particular, when only observing the OT plans \ref{enum:ot_plan}, we cannot achieve more than identifiability in $\shiftc^{N \times M}$. We note that the linear space of shifts of the form $a \oplus b$ has dimension $N + M - 1$.

\begin{remark}[Generalization to Linear Programs] \label{rem:lin_prog}
    The considerations in this section are a particular case of the more general linear program \cite{Bertsimas1997,Luenberger2008} with fixed (and known) constraint matrix $A \in \R^{E \times D}$: For $\vc{c} \in \R^{D}$, $\vc{b} \in \R^{E}$ consider the primal problem
    \begin{equation*}
         \LP_{\vc{c}}(\vc{b}) \defeq \min_{\vc{x} \in P(\vc{b})} \inner{\vc{c}}{  \vc{x} }\,,
    \end{equation*}
    where the polytope $P(\vc{b}) \defeq \{ \vc{x} \in \R^D : A\vc{x}= \vc{b},\, \vc{x} \geq 0 \}$ is bounded and non-empty. The corresponding dual problem reads
	\begin{equation*}
		\LP_{\vc{c}}(\vc{b}) = \max_{\vc{y} \in P^*(\vc{c})} \inner{\vc{b}}{\vc{y}} \,,
	\end{equation*}
	where $P^*(\vc{c}) \defeq \{ \vc{y} \in \R^{E} : A^\transp \vc{y} \leq \vc{c} \}$.

    Here, we observe $K$ vectors $\vc{b}^{(k)} \in \R^{E}$, $k \in \idn{K}$, with corresponding primal and dual solutions $\vc{x}^{(k)} \in P_A(\vc{b^{(k)}})$ and $\vc{y}^{(k)} \in P^*(\vc{c})$ as well as the optimal values $\alpha^{(k)} \in \R$, i.e.,
	\begin{equation*}
		\alpha^{(k)} = \LP_{\vc{c}}(\vc{b}^{(k)}) = \inner{\vc{c}}{\vc{x}^{(k)}} = \inner{\vc{b}^{(k)}}{\vc{y}^{(k)}}\,.
	\end{equation*}
	From (a subset of) this information we want to uniquely determine the underlying $\vc{c}$.

	Strong duality yields the analog to \autoref{lemma:primal_dual_opt}: $\vc{x} \in P(\vc{b})$ and $\vc{y} \in P^*(\vc{c})$ are optimal if and only if $\inner{\vc{c}}{\vc{x}} = \inner{\vc{b}}{\vc{y}}$, or equivalently if $\vc{x}_i > 0$ for some $i \in \idn{D}$ implies that $(A^\transp \vc{y})_i = \vc{c}_i$.

    The more general formulation of shift equivalence is as follows: Write $\vc{c} \sim_{A^\transp} \vc{c'}$ if $\vc{c} - \vc{c'} = A^\transp \vc{y}$ for some $\vc{y} \in \R^E$. As it holds for all $\vc{x} \in P(\vc{b})$ that
    \begin{equation*}
        \inner{\vc{c} + A^\transp\vc{y}}{\vc{x}} = \inner{\vc{c}}{\vc{x}} + \inner{\vc{y}}{A \vc{x}} = \inner{\vc{c}}{\vc{x}} + \inner{\vc{y}}{\vc{b}}\,,
    \end{equation*}
    we see that the shift of $\vc{c}$ by $A^\transp \vc{y}$ does not change the set of minimizers of the linear program. Note that these shifts lie in the image of the linear operator $A^\transp$.
\end{remark}

Finally, we introduce the (to OT specific) concept of Monge matrices (for an overview, see \cite{burkard1996perspectives}). A cost matrix $c \in \R^{N \times M}$ satisfies the so-called Monge property if it holds that
\begin{equation} \label{eq:monge_prop}
	c_{i,j} + c_{r,s} \leq c_{i,s} + c_{r,j} \qquad \text{for all } 1 \leq i < r \leq N,\, 1 \leq j < s \leq M\,.
\end{equation}

\begin{example}[Monge matrices] \label{ex:monge_mat}
	Let $h : \R \to \R$ be convex and $\{ x_1, \ldots, x_N \}$, $\{y_1, \ldots, y_M \}$ be increasingly sorted sets of points in $\R$. Then, the induced cost matrix with $c_{i,j} = h(x_i - y_j)$ for all $(i, j) \in \idn{N} \times \idn{M}$ satisfies the Monge property. This includes cost matrices of the form $c_{i,j} = \abs{x_i - y_j}^p$ for $p \geq 1$.
\end{example}

When the cost matrix $c$ is Monge, the north-west-corner rule always yields an OT plan \cite{Hoffman1963}. Hence, observing $K$ marginals, we can always construct OT plans $\pi^{(k)}$, $k \in \idn{K}$. To this end, denote the cumulative sums of $\mu^{(k)}$ with
\begin{equation} \label{eq:cusum_monge}
	A^{(k)}_i \coloneqq \sum_{r=1}^{i} \mu^{(k)}_r\,, \qquad i \in \{ 0 \} \cup \idn{N}\,,
\end{equation}
and define $\{ B^{(k)}_j \}_{j=0}^M$ similarly as the cumulative sums of $\nu^{(k)}$. Then, an OT plan $\pi^{(k)}$ between $\mu^{(k)}$ and $\nu^{(k)}$ w.r.t.\ a Monge matrix $c$ is given by the monotone plan
\begin{equation} \label{eq:monotone_plan}
	\pi^{(k)}_{i,j} = \lambda_1\{ ( A_{i-1}^{(k)}, A_i^{(k)}] \cap (B_{j-1}^{(k)}, B_j^{(k)}] \} \,, \qquad i \in \idn{N} ,\, j \in \idn{M}\,,
\end{equation}
where $\lambda_1$ is the one-dimensional Lebesgue measure. As this does not directly depend on $c$, only on $c$ being Monge, we again see that knowledge of the total cost \ref{enum:ot_cost} is crucial for identifiability.

\subsection{Identifiability from Total Costs Alone}\label{Sec:TotalCost}

First, we consider the most difficult case of only observing the total costs \ref{enum:ot_cost},
\begin{equation} \label{eq:costs_known}
	\alpha^{(k)} = \OT_c(\mu^{(k)}, \nu^{(k)}) \,, \qquad k \in \idn{K} \,.
\end{equation}
Here, we only know that for each $k \in \idn{K}$ there exists at least one extreme point $u^{(k)} \in \EP(\mu^{(k)}, \nu^{(k)})$ such that $\inner{c}{u^{(k)}} = \alpha^{(k)}$ and any other $v^{(k)} \in \EP(\mu^{(k)}, \nu^{(k)})$ must satisfy $\inner{c}{v^{(k)}} \geq \alpha^{(k)}$. This leads to a multitude of LPs whose solutions, if existent, all give a valid cost matrix.

\begin{theorem} \label{thm:ident_only_total}
	Suppose that we observe $K$ distinct pairs of marginals and the corresponding total costs \ref{enum:ot_cost}. Then, $c$ is identifiable in $\R^{N \times M}$ if and only if for all possible combinations of $u^{(k)} \in \EP(\mu^{(k)}, \nu^{(k)})$, $k \in \idn{K}$, the system
	\begin{equation} \label{eq:system_only_total}
		\begin{aligned}
			\inner{c}{v^{(k)}} &\geq \alpha^{(k)}, &&& v^{(k)} &\in \EP(\mu^{(k)}, \nu^{(k)}) \setminus \{ u^{(k)} \},  \\
			\inner{c}{u^{(k)}} &= \alpha^{(k)}, &&& k &\in \idn{K},
		\end{aligned}
	\end{equation}
	is either inconsistent or the solution is unique and independent of the choice. Note that if a solution exists (that is allowed to depend on the choice), then we have existence of the underlying cost matrix $c$.
\end{theorem}
\begin{proof}
	A cost matrix $c$ is consistent with the observed OT information if and only if it satisfies the above system \eqref{eq:system_only_total} for at least one combination.
\end{proof}

\begin{remark}[Additional linear constraints on $c$]
	Each system \eqref{eq:system_only_total} in \autoref{thm:ident_only_total} is an LP. This formulation is quite flexible as it allows us to place additional linear (in-)equality constraints on the cost matrix $c$. However, the LP might not be proper in the sense that the feasibility set can be unbounded. In general, this can be dealt with by placing a boundedness restriction on the set of cost matrices $c$, e.g., $0 \leq c \leq C_0$ for some constant $C_0 > 0$.
\end{remark}

\begin{remark}[Set of consistent cost matrices]
    As each system \eqref{eq:system_only_total} in \autoref{thm:ident_only_total} is an LP, it follows that the set of consistent cost matrices is equal to a union of at most $K$ (possibly unbounded) polytopes.
\end{remark}

\begin{remark}[Reformulation as Mixed Integer Linear Program]
	Note that the formulation involving all possible combinations of optimal extreme points in \autoref{thm:ident_only_pot} can be reformulated as follows: The cost matrix $c$ is identifiable in $\R^{N \times M}$ if and only if the system
	\begin{equation*}
		\inner{c}{v^{(k)}} \geq \alpha^{(k)} \,, \qquad v^{(k)} \in \EP(\mu^{(k)}, \nu^{(k)}), \qquad k \in \idn{K},
	\end{equation*}
	where equality is attained for at least one $v^{(k)}$ for each $k \in \idn{K}$, has a unique solution.

	Assuming that $c$ is bounded by $C_0 > 0$, this can be formulated as a mixed integer linear program (MILP), i.e., an LP where some variables are constrained to be integer. To this end, enumerate for all $k \in \idn{K}$ the extreme points $\EP(\mu^{(k)}, \nu^{(k)})$ as $\{ v^{(1, k)}, \ldots, v^{(R_k, k)} \}$ and introduce for $j \in \idn{R_k}$, $k \in \idn{K}$ the binary variables $e^{(j, k)} \in \{0, 1\}$ and the slack variables $z^{(j, k)} \in \R$. Then, the MILP reads
	\begin{align*}
		\sum_{r=1}^{R_k} e^{(r, k)} &\geq 1, &&& 0 \leq z^{(j, k)} &\leq [C_0 - \alpha^{(k)}] (1 - e^{(j, k)}), \\
 		\inner{c}{v^{(j, k)}} - z^{(j, k)} &= \alpha^{(k)}, &&& j &\in \idn{R_k},\; k \in \idn{K}.
	\end{align*}
	Note that the free variables are $c$ and $e^{(j, k)},\, v^{(j, k)}$ with $j \in \idn{R_k}$, $k \in \idn{K}$. In particular, $c$ is identifiable in $[0, C_0]^{N \times M}$ if and only if the MILP has a unique solution in $c$ (not the full set of free variables). Since solving MILPs is NP-hard, in general, this suggests that inferring the cost matrix from marginals and total costs only is a difficult and computationally expensive problem. The situation simplifies significantly if the OT plans are observed, see \autoref{Sec:Estimation}. Furthermore, note that in the current formulation, any $e^{(k,r)} = 1$ can be changed to $0$ if $\sum_{j=1}^{R_k} e^{(j, k)} \geq 2$, without changing the solution in $c$. Hence, uniqueness of the full solution of the MILP is stronger than identifiability of $c$. In fact, the MILP can only have a unique solution when the underlying optimal faces each consist of a single point.
\end{remark}

\begin{remark}[Computational complexity] \label{rem:comp_complexity_total}
	Denote $V_k = \abs{\EP(\mu^{(k)}, \nu^{(k)})}$ with $V \defeq \prod_{k=1}^K V_k$ and $W \defeq \sum_{k=1}^K V_k$. Then, to check identifiability of $c$ via \autoref{thm:ident_only_total} we have to determine the unique solvability of $V$ LPs, each with variable size $NM$ and $W$ constraints. Whether a given LP has a unique solution can be checked by computing an adjusted version of the LP, see \autoref{rem:comp_complexity_total_plan} for more details. As the $V_k$ can scale exponentially in $N$ and $M$ (see also \autoref{rem:comp_ri}) the computational burden of \autoref{thm:ident_only_total} can be enormous and practically feasible only for small values of $N$ and $M$.
\end{remark}

\begin{remark}[Simpler sufficient condition] \label{rem:ident_only_total_eq}
	By dropping the inequalities in the system \eqref{eq:system_only_total} given in \autoref{thm:ident_only_total}, we get a sufficient condition for identifiability that is easier to check as it only involves linear equations. Namely, $c$ is identifiable in $\R^{N \times M}$ if for all possible combinations of potentially optimal $u^{(k)} \in \EP(\mu^{(k)}, \nu^{(k)})$, $k \in \idn{K}$, the system
	\begin{equation} \label{eq:system_only_total_eq}
		\inner{c}{u^{(k)}} = \alpha^{(k)},\qquad k \in \idn{K},
	\end{equation}
	is either inconsistent or the solution is unique and independent of the choice. Note that here it is crucial that we assume that the underlying cost matrix $c$ exists, as the above system (in contrast to the one in \autoref{thm:ident_only_total}) does not ensure consistency, see also \autoref{ex:incons_c_lin_eq}. Finally, this simpler condition decreases the computational complexity (in comparison to \autoref{rem:comp_complexity_total}) to $V$ Gaussian eliminations in $\bigO( [NM]^2 K)$ time.
\end{remark}

We now give an example where \autoref{thm:ident_only_total} provides identifiability but not the conditions given in \autoref{rem:ident_only_total_eq}.

\begin{example}
	Suppose that for $N = M = 2$ we observe the following OT information consisting of $K = 4$ pairs of marginals and corresponding total costs
    \begin{align*}
		\mu^{(1)} &=(3/4, 1/4),&\mu^{(2)} &=(3/7, 4/7),&\mu^{(3)} &=(4/5, 1/5),&\mu^{(4)} &=(3/7, 4/7), \\
		\nu^{(1)} &=(5/8, 3/8),&\nu^{(2)} &=(1/5, 4/5),&\nu^{(3)} &=(1/2, 1/2),&\nu^{(4)} &=(1/2, 1/2), \\
		\alpha^{(1)} &=7/4,&\alpha^{(2)} &=1,&\alpha^{(3)} &=1,&\alpha^{(4)} &=1.
	\end{align*}
	Then, it can be checked that \autoref{thm:ident_only_total} yields identifiability of the underlying cost matrix $c$ in $\R^{N \times M}$, and it is given by
	\begin{equation*}
		c = \begin{bmatrix} 9/2 & -2 \\ 13/4 & 13/4 \end{bmatrix}.
	\end{equation*}
	To be more precise, only the following combination of extreme points yields a solution to the system \eqref{eq:system_only_total}:
    \begin{align*}
        u^{(1)} &= \begin{bmatrix} 3/8 & 3/8 \\ 1/4 & 0 \end{bmatrix}, &  u^{(2)} &= \begin{bmatrix} 0 & 3/7 \\ 1/5 & 13/35 \end{bmatrix}, \\
        u^{(3)} &= \begin{bmatrix} 3/10 & 1/2 \\ 1/5 & 0 \end{bmatrix}, &
        u^{(4)} &= \begin{bmatrix} 0 & 3/7 \\ 1/2 & 1/14 \end{bmatrix}.
    \end{align*}
    All other combinations of extreme points lead to an inconsistent system \eqref{eq:system_only_total}. Notably, when dropping the inequality constraints, each corresponding system \eqref{eq:system_only_total_eq} has a unique solution that is not independent of the combination. Hence, in this example \autoref{rem:ident_only_total_eq} cannot be used to obtain identifiability.
\end{example}

If we know that the cost matrix satisfies the Monge property \eqref{eq:monge_prop}, we already know one extreme point that is optimal. Hence, in this case identifiability reduces to the unique solvability of a single LP.

\begin{corollary} \label{cor:ident_only_total_monge}
	Suppose that we observe $K$ distinct pairs of marginals and the corresponding total costs \ref{enum:ot_cost}. For each $k \in \idn{K}$ denote with $\pi^{(k)} \in \EP(\mu^{(k)}, \nu^{(k)})$ the monotone transport plan \eqref{eq:monotone_plan}. Then, $c$ is identifiable in the class of Monge matrices if and only if the system \eqref{eq:system_only_total} with $u^{(k)} = \pi^{(k)}$, $k \in \idn{K}$, has a unique solution. Note that if a solution exists that satisfies the Monge property, then we have existence of the underlying cost matrix $c$.
\end{corollary}

Note that the Monge property \eqref{eq:monge_prop} is a linear inequality constraint on the cost matrix $c$. Hence, the Monge matrix constraint can be explicitly added to the system \eqref{eq:system_only_total} of \autoref{cor:ident_only_total_monge} without changing the LP structure. The computational complexity to determine in \autoref{cor:ident_only_total_monge} whether the system \eqref{eq:system_only_total} has a unique solution is similar to the one in the case that we observe the total costs as well as the OT plans, which we discuss in more detail in \autoref{rem:comp_complexity_total_plan}.

\subsection{Identifiability from Optimal Potentials Alone}\label{Sec:OTpotent}

Next, we consider the case where we only know the  optimal potentials \ref{enum:ot_pot}. Note that, in such a scenario, the total costs \ref{enum:ot_cost}  are determined by the relations
\begin{equation} \label{eq:pot_known}
	\alpha^{(k)} = \OT_c(\mu^{(k)}, \nu^{(k)}) = \inner{f^{(k)}}{\mu^{(k)}} + \inner{g^{(k)}}{\nu^{(k)}}\,, \qquad k \in \idn{K} \,.
\end{equation}

We proceed as in \autoref{thm:ident_only_total} and consider all combinations of possible optimal extreme points. Using the primal-dual optimality criterion (\autoref{lemma:primal_dual_opt}), we get a more restrictive system.

\begin{theorem} \label{thm:ident_only_pot}
	Suppose that we observe $K$ distinct pairs of marginals and the corresponding total costs \ref{enum:ot_cost} and optimal potentials \ref{enum:ot_pot}. Then, $c$ is identifiable in $\R^{N \times M}$ if and only if for all possible combinations of $u^{(k)} \in \EP(\mu^{(k)}, \nu^{(k)})$, $k \in \idn{K}$, the system
	\begin{equation} \label{eq:system_only_pot}
		\begin{aligned}
			\inner{c}{v^{(k)}} &\geq \alpha^{(k)}, &&& \inner{c}{u^{(k)}} &= \alpha^{(k)}, \\
			c &\geq f^{(k)} \oplus g^{(k)} \text{ with equality on } \supp u^{(k)}, &&& k &\in \idn{K},
			\\ v^{(k)} &\in \EP(\mu^{(k)}, \nu^{(k)}) \setminus \{ u^{(k)} \},
		\end{aligned}
	\end{equation}
	is either inconsistent or the solution is unique and independent of the choice. Note that if a solution exists (that is allowed to depend on the choice), then we have existence of the underlying cost matrix $c$.
\end{theorem}

Note that the computational complexity is similar to the condition to be checked in \autoref{thm:ident_only_total}, recall \autoref{rem:comp_complexity_total}. Additionally, we point out that the condition $c = f^{(k)} \oplus g^{(k)}$ on $\supp u^{(k)}$ directly decreases the number of free variables.

For the class of Monge matrices, we can again use the structural knowledge of the OT to arrive at the following more specialized version of \autoref{thm:ident_only_pot}.

\begin{corollary} \label{cor:ident_only_pot_monge}
	Suppose that we observe $K$ pairs of marginals and the corresponding total costs \ref{enum:ot_cost} and optimal potentials \ref{enum:ot_pot}. For each $k \in \idn{K}$ denote with $\pi^{(k)} \in \EP(\mu^{(k)}, \nu^{(k)})$ the monotone transport plan \eqref{eq:monotone_plan}. Then, the cost matrix $c$ is identifiable in the class of Monge matrices if and only if the system \eqref{eq:system_only_pot} with $u^{(k)} = \pi^{(k)}$, $k \in \idn{K}$, has a unique solution. Note that if a solution exists that satisfies the Monge property, then we have existence of the underlying cost matrix.
\end{corollary}

Again, checking the above condition corresponds to determining whether an LP has a unique solution (see \autoref{rem:comp_complexity_total_plan}). Furthermore, only using the primal-dual optimality criterion (\autoref{lemma:primal_dual_opt}), we obtain the following easily verifiable sufficient condition for the identifiability of Monge matrices. Recall that this encompasses all the matrices given in \autoref{ex:monge_mat}.

\begin{corollary} \label{cor:ident_only_pot_monge_primaldual}
	Suppose that we observe $K$ distinct pairs of marginals and the corresponding total costs \ref{enum:ot_cost} and optimal potentials \ref{enum:ot_pot}. Recall the definition of $A^{(k)}_i$ and $B^{(k)}_j$ in \eqref{eq:cusum_monge}. Then, $c$ is identifiable in the class of Monge matrices if for all $i \in \idn{N}$, $j \in \idn{M}$ there exists a $k \in \idn{K}$ such that
	\begin{equation*}
		\max(A_{i-1}^{(k)}, B_{j-1}^{(k)}) < \min(A_i^{(k)}, B_j^{(k)}) \,.
	\end{equation*}
\end{corollary}

This identifiability criterion essentially means computing the monotone plans \eqref{eq:monotone_plan} which can be done in $\bigO(K(N + M))$ elementary operations (see \cite{Hoffman1963}). Next, we give an example where \autoref{cor:ident_only_pot_monge}, but not \autoref{cor:ident_only_pot_monge_primaldual}, yields identifiability.

\begin{example}
	Suppose that for $N = M = 2$ we observe the following OT information consisting of $K = 3$ pairs of marginals and corresponding total costs and optimal potentials
	\begin{align*}
		\mu^{(1)} &= (1/2, 1/2),&&&\mu^{(2)} &= (1/3, 2/3),&&&\mu^{(3)} &= (1, 0),\\
		\nu^{(1)} &= (1, 0),&&&\nu^{(2)} &= (1, 0),&&&\nu^{(3)} &= (2/5, 3/5),\\
		\alpha^{(1)} &= 1,&&&\alpha^{(2)} &= 2/3,&&&\alpha^{(3)} &= 5/3,\\
		f^{(1)} &= (0, -2),&&&f^{(2)} &= (0, -2),&&&f^{(3)} &= (2, 0),\\
		g^{(1)} &= (2, 0),&&&g^{(2)} &= (2, 0),&&&g^{(3)} &= (0, -5/9).
	\end{align*}
	Then, \autoref{cor:ident_only_pot_monge} yields that the following underlying cost matrix
	\begin{equation*}
		c = \begin{bmatrix} 2 & 13/9 \\ 0 & -5/9 \end{bmatrix}
	\end{equation*}
	is identifiable in the class of Monge matrices. However, the given OT information is not enough to obtain identifiability via \autoref{cor:ident_only_pot_monge_primaldual}.
\end{example}

\subsection{Identifiability from Optimal Transport Plans Alone}\label{Sec:Plans}

This subsection is devoted to the identifiability of the cost matrix when only the OT plans \ref{enum:ot_plan} are observed. This is quite different from the other settings as due to the shift invariance \eqref{eq:shift_equiv} we cannot achieve true identifiability. Nevertheless, we will be able to identify the underlying cost matrix to some extent.

\begin{figure}
    \centering
\includegraphics[width=0.9\textwidth]{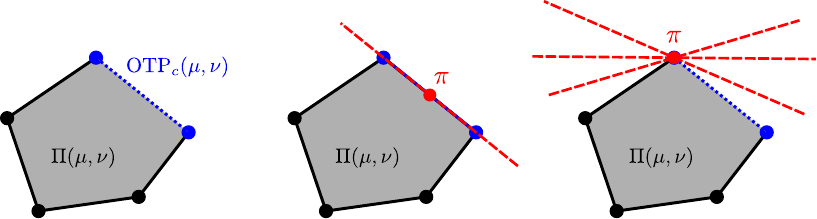}
    \caption{The polytope $\Pi(\mu, \nu)$ (gray-filled) with optimality set $\OPT_c(\mu, \nu)$ (dotted line): The observation of a point in the relative interior of the optimality set provides more information than the knowledge of an extreme point. Middle figure: Only one possible tangent plane (dashed line) at the point $\pi$. Right figure: A whole bundle of hyperplanes (dashed lines). Therefore, in the figure in the middle, we have more information to identify the cost matrix.}
    \label{fig:visualExplanation}
\end{figure}

Observing the optimal plans \ref{enum:ot_plan}, we know by \autoref{pr:ri} that the smallest face that $\pi^{(k)}$ lies on must be contained in $\OPT_c(\mu^{(k)}, \nu^{(k)})$. This yields further information about the tangent hyperplane $\pi^{(k)} +\{ x \in \R^{N \times M} : \inner{x}{c} = 0 \}$, see \autoref{fig:visualExplanation}. Depending on the size of $\spanop \{ \pi^{(k)} - u^{(\ell, k)} : k \in \idn{K},\, \ell \in \idn{L_k}\}$, we will see that the cost matrix may be identifiable in a suitable sense. Recall that $\shiftc^{N \times M}$ denotes the quotient space of cost matrices up to shift equivalence \eqref{eq:shift_equiv}. In the following theorem, we show that we can identify $c$ in $\shiftc^{N \times M}$ up to a span of a number of linearly independent vectors.

\begin{theorem} \label{thm:ident_only_plans}
    Suppose that we observe $K$ distinct pairs of marginals and the corresponding OT plans \ref{enum:ot_plan}. Let $\{ u^{(\ell, k)} : \ell \in \idn{L_k} \}$, $k \in \idn{K}$, be as in \autoref{pr:ri}. Then, the following holds:
    \begin{enumerate}
        \item Let $S \in \N_0$ be such that
        \begin{equation*}
 	          \dim ( \spanop \{ \pi^{(k)} - u^{(\ell,k)} : \ell \in \idn{L_k},\,k \in \idn{K} \} ) = (N-1)(M-1) - S,
        \end{equation*}
        then, the cost matrix $c$ is identifiable in $\shiftc^{N \times M}$ up to the span of $S$ linearly independent vectors.
        \item If $\coh \{u^{(\ell, k)} : \ell \in \idn{L_k} \} =\OPT_c(\pi^{(k)})$ for all $k \in \idn{K}$, then the converse holds.
    \end{enumerate}
\end{theorem}

To prove this, we first use that by optimality the cost matrix $c$ must lie in the orthogonal complement of $\spanop \{ \pi^{(k)} - u^{(\ell,k)} : \ell \in \idn{L_k},\,k \in \idn{K} \}$. Note that this complement has dimension $S + (N + M  - 1)$, where $N + M - 1$ of these are generated by shift equivalence \eqref{eq:shift_equiv}, i.e., $\{ a \oplus b : a \in \R^N,\; b \in \R^M\}$. As discussed before, we cannot hope to identify these $N + M - 1$ dimensions via OT plans alone, which leaves us with the remaining $S$ dimensions.

For the converse direction, note that $\coh \{u^{(\ell, k)} : \ell \in \idn{L_k} \} =\OPT_c(\pi^{(k)})$ means that we observe the full optimal face of $\Pi(\pi^{(k)})$ for all $k \in \idn{K}$. Intuitively, if this is not the cases, then the orthogonal complement of $\spanop \{ \pi^{(k)} - u^{(\ell,k)} : \ell \in \idn{L_k},\,k \in \idn{K} \}$ contains a direction that is consistent with the observed OT plans. We use this idea to show that if we perturb $c$ by any $c'$ in the orthogonal complement, i.e., we consider $c_{\tau} = c + \tau c'$, then eventually $c_{\tau}$ will always be a consistent cost matrix as $\tau \to 0$.

\begin{proof}
For $k \in \idn{K}$ it holds by $\pi^{(k)}, u^{(1,k)}, \ldots, u^{(L_k, k)} \in \OPT_c(\pi^{(k)})$ that
\begin{equation*}
 	 \inner{c}{\pi^{(k)}} = \inner{c}{u^{(\ell, k)}} \qquad \text{for all } \ell \in \idn{L_k}\,.
 \end{equation*}
This implies that
\begin{equation*}
    \spanop \left\{ \pi^{(k)} - u^{(\ell,k)}: \ell \in \idn{L_k},\,k \in \idn{K} \right\} \perp c
\end{equation*}
and as a consequence,
\begin{equation*}
    c \in \left[ \spanop \left\{ \pi^{(k)} - u^{(\ell,k)}: \ell \in \idn{L_k},\,k \in \idn{K} \right\} \right]^\perp\,.
\end{equation*}
By assumption, this orthogonal complement has dimension $$NM - (N-1)(M-1) + S = N + M + S - 1.$$ We show that $N + M - 1$ of these dimensions are spanned by $\{ a \oplus b : a \in \R^N,\; b \in \R^M\}$. To this end, we prove that $\left\{ \pi^{(k)} - u^{(\ell,k)} : \ell \in \idn{L_k},\,k \in \idn{K} \right\}$ and $\{ a \oplus b : a \in \R^N, \; b \in \R^M\}$ are linearly independent. For a fixed $k \in \idn{K}$, suppose that there exists scalars $\lambda^{(k)}_\ell \in \R$, $\ell \in \idn{L_k}$, and $R \in \N$ pairs $(a^{(r)}, b^{(r)}) \in \R^{N} \times \R^{M}$ with scalars $\tau_r \in \R$, $r \in \idn{R}$, such that
\begin{equation*}
    v \defeq \sum_{k=1}^K \sum_{\ell=1}^{L_k} \lambda^{(k)}_\ell [\pi^{(k)} - u^{(\ell,k)}] = \sum_{r=1}^R \tau_r [a^{(r)} \oplus b^{(r)}]\,.
\end{equation*}
Then, it follows that
\begin{align*}
    \norm{v}_2^2 &= \inner*{ \sum_{k=1}^K \sum_{\ell=1}^{L_k} \lambda^{(k)}_\ell [\pi^{(k)} - u^{(\ell,k)}] }{ \sum_{r=1}^R \tau_r [a^{(r)} \oplus b^{(r)}] } \\
    &= \sum_{k=1}^K \sum_{\ell=1}^{L_k} \sum_{r=1}^R \lambda^{(k)}_\ell \tau_r \inner{\pi^{(k)} - u^{(\ell,k)}}{a^{(r)} \oplus b^{(r)}} \\
    &= \sum_{k=1}^K \sum_{\ell=1}^{L_k} \sum_{r=1}^R \lambda^{(k)}_\ell \tau_r \left[ \sum_{i=1}^N a_i^{(r)} (\mu^{(k)}_i - \mu^{(k)}_i) + \sum_{j=1}^M b_j^{(r)} ( \nu^{(k)}_j - \nu^{(k)}_j)   \right] = 0\,,
\end{align*}
which implies $v = 0$ and \textit{a fortiori}  $\left\{ \pi^{(k)} - u^{(\ell,k)} : \ell \in \idn{L_k},\,k \in \idn{K} \right\}$ and $\{ a \oplus b : a \in \R^N, \; b \in \R^M\}$ are linearly independent. Hence, $c$ is determined in $\shiftc^{N \times M}$ up to the remaining $S$ dimensions.

For the converse, assume that
 \begin{equation*}
 	 \dim (\spanop \{ \pi^{(k)} - u^{(\ell,k)} : \ell \in \idn{L_k},\,k \in \idn{K} \}) < (N-1)(M-1) - S \,.
 \end{equation*}
Then, there exists more than $S$ linearly independent vectors, so that there exist  $c'$ such that
\begin{equation*}
	c, c'\in \left[ \spanop \left\{ \pi^{(k)} - u^{(\ell,k)} : \ell \in \idn{L_k},\,k \in \idn{K} \right\} \right]^{\perp}\,.
\end{equation*}
Then,
\begin{equation} \label{eq:c_c_prime_inner_zero}
    \inner{c}{\pi^{(k)} - u^{(\ell,k)}} = 0 = \inner{c'}{\pi^{(k)} - u^{(\ell,k)}} \qquad \text{for all } \ell \in \idn{L_k},\,k \in \idn{K}.
\end{equation}
Hence,
\begin{equation*}
	\inner{c}{\pi^{(k)}} = \inner{c}{u^{(\ell,k)}} \quad {\rm and} \quad \inner{c'}{\pi^{(k)}} = \inner{c'}{u^{(\ell,k)}} \qquad \text{for all } \ell \in \idn{L_k},\, k \in \idn{K}.
\end{equation*}
We claim that there exists $ \tau(k)>0$ such that for all $ 0\leq \tau\leq \tau(k)  $ it holds that $\pi^{(k)} \in \OPT_{c_\tau}(\pi^{(k)})$ where $c_\tau = c + \tau c'$. We prove it by contradiction. Assume that for all $ \tau(k)>0$ there exists $ 0\leq \tau\leq \tau(k)  $ such that $\pi^{(k)} \notin \OPT_{c_\tau}(\pi^{(k)})$. Then, we can create a sequence of positive numbers  $\tau_{s} > 0$ with $\tau_s\to 0$ for $s \to \infty$ and a sequence of couplings $\gamma_s \in \OPT_{c_{\tau_s}}(\pi^{(k)}) \cap \EP(\pi^{(k)})$ such that
\begin{equation}
    \label{eq:contradictionTheoremIdentOnlyPlans}
     \inner{c_{\tau_s}}{\pi^{(k)}} > \inner{c_{\tau_s}}{\gamma_s}.
\end{equation}
We can further assume (subtracting, if necessary, a further subsequence) that $\gamma_s\to \gamma \in \Pi(\pi^{(k)})$ as $s\to \infty$. Since the set of extreme points is finite, it follows that $\gamma \in \EP(\pi^{(k)})$ and $\gamma_s = \gamma$ for all $s \geq s_0$ for some $s_0 \in \N$. Since, for any $\pi \in \Pi(\pi^{(k)})$ it holds that $\inner{c_{\tau_s}}{\gamma_s} \leq \inner{c_{\tau_s}}{\pi}$, we have
\begin{align*}
    \inner{c_{\tau_s}}{\gamma_s-\gamma} + \inner{c_{\tau_s}}{\gamma} \leq \inner{c_{\tau_s}}{\pi}.
\end{align*}
Taking limits, we get
 \begin{align*}
    \inner{c}{\gamma} \leq \inner{c}{\pi}, \qquad \text{ for all } \pi \in \Pi(\pi^{(k)})\,,
\end{align*}
so that $\gamma \in \OPT_c(\pi^{(k)}) \cap \EP(\pi^{(k)})$. As $\coh\left(\{u^{(\ell, k)} : \ell \in \idn{L_k} \} \right)=\OPT_c(\pi^{(k)})$, we can assume without loss of generality that $\gamma= u^{(1, k)}$. We notice that in such a case
\begin{equation*}
	\inner{c}{\pi^{(k)} - u^{(1,k)}} = 0.
\end{equation*}
Using \eqref{eq:c_c_prime_inner_zero}, we have for all $s \geq s_0$ that
\begin{equation*}
    \inner{c_{\tau_s}}{\gamma_s} = \inner{c}{u^{(1,k)}} + \tau_s \inner{c'}{u^{(1,k)}} = \inner{c}{\pi^{(k)}} + \tau_s \inner{c'}{\pi^{(k)}} = \inner{c_{\tau_s}}{\pi^{(k)}}\,,
\end{equation*}
which contradicts \eqref{eq:contradictionTheoremIdentOnlyPlans}. Therefore, the claim holds. Then, for any $k\in \idn{K}$, the observed plan $\pi^{(k)}$ is optimal for any of cost matrix $c_{\tau}=c+\tau\, c'$, with $0\leq \tau\leq \min \{\tau(1), \ldots, \tau(K)\}$, so that identifiability cannot hold in $\shiftc^{N \times M}$ up to the span of $S$ linearly independent vectors.
\end{proof}

\begin{remark}[The case $S = 0$] \label{rem:pi_known_s_0}
    In \autoref{thm:ident_only_plans} with $S = 0$, the given dimension of $(N-1)(M-1)$ can be attained. Namely, a single transport polytope can have this affine dimension \cite{Deloera2013}. Note that in this case it follows for the underlying cost matrix that $c \shifteq 0$.
\end{remark}

\begin{remark}[The case $S = 1$] \label{rem:pi_known_s_1}
    In \autoref{thm:ident_only_plans}, if $S = 1$ then the cost matrix $c$ is determined up to shift equivalence and multiplication by a scalar, i.e., $c$ can (up to shift equivalence) be written as $c = \tau c'$ for $\tau > 0$ and some cost matrix $c'$ that has the right sign (determined such that all the observed OT plans are optimal for $c'$). If we care only about calculating the OT plans, then the true value of $\tau$ does not matter.
\end{remark}

\begin{remark} \label{rem:inv_feas_reg}
	A related result to \autoref{thm:ident_only_plans} is given in \cite[Proposition~6]{Tavaslioglu2018}. They characterize the inverse-feasibility region for $K = 1$ of an LP (what corresponds to the set of consistent cost matrices here) as the Minkowski sum of a cone and span, both depending on the active constraints of the LP. More specifically, in our setting they assert that every consistent cost matrix $c$ for the observed OT plan $\pi^{(1)}$ is of the form
	\begin{equation} \label{eq:cone_span}
		c = \sum_{(i,j) \in [\supp \pi^{(1)}]^\complement} \lambda_{i,j} e^{i,j} + a \oplus b\,, \qquad \lambda \in \R^{N \times M}_+,\, a \in \R^N, \, b \in \R^M\,,
	\end{equation}
	where $[\supp \pi^{(1)}]^\complement = \idn{N} \times \idn{M} \setminus \supp \pi^{(1)}$ is the complement of the support and $e^{i,j} \in \R^{N \times M}$ has the entry $1$ at $(i, j)$ and zero else. In particular, this implies that if $[\supp \pi^{(1)}]^\complement = \emptyset$, then $c \shifteq 0$ as in \autoref{rem:pi_known_s_0}. Similarly, if $[\supp \pi^{(1)}]^{\complement} = \{ (i, j) \}$ for one pair $(i, j) \in \idn{N} \times \idn{M}$, then $c \shifteq e^{i,j}$ and \autoref{rem:pi_known_s_1} applies.

	Furthermore, in the case $K > 1$, it follows that the set of consistent cost matrices is the intersection of cost matrices of the form \eqref{eq:cone_span} where $\pi^{(1)}$ is substituted by $\pi^{(k)}$, $k\in\idn{K}$. Due to the additive structure, a representation like \eqref{eq:cone_span} for these intersections is not straightforward or might not even hold. This makes it challenging to extend \eqref{eq:cone_span} to general $K$ in order to derive identifiability.
\end{remark}

\begin{remark}[Computation of the span containing $c$] \label{rem:comp_complexity_plan}
    The proof of \autoref{thm:ident_only_plans} reveals how the span which contains $c$ (up to shift equivalence) can be calculated from the observed OT plans. In fact, this can be implemented by solving a linear system of equations. Let
    \begin{equation*}
        E^{(k)} \defeq \begin{pmatrix}
            [\pi^{(k)} - u^{(1,k)}]^\transp \\
            \vdots \\
            [\pi^{(k)} - u^{(L_k,k)}]^\transp \\
        \end{pmatrix}\,, \qquad E_K \defeq \begin{pmatrix}
            E^{(1)} \\ \vdots \\ E^{(K)}
        \end{pmatrix}\,.
    \end{equation*}
    Then, $c$ must satisfy
    \begin{equation*}
        E_K c = 0\,.
    \end{equation*}
	This is a system of $L \defeq \sum_{k=1}^K L_k$ equations of size $NM$ and can be solved with Gaussian elimination in $\bigO(\min(L, NM) L NM)$ time. Furthermore, recall that in a preliminary step the extreme points from \autoref{pr:ri} need to be computed, see \autoref{rem:comp_ri} for its computational complexity. Assuming $N = M$ and $L \geq N^2$, we get overall complexity of $\bigO(N^4 L)$. This can be quite computationally intensive in $N$, in particular as $L$ can potentially scale exponentially in $N$.
\end{remark}

\subsection{Identifiability from Total Costs and Optimal Transport Plans}\label{Sec:planCost}

Now, we consider the case where we only observe the total costs \ref{enum:ot_cost} and OT plans \ref{enum:ot_plan}, i.e.,
\begin{equation} \label{eq:total_costs_plans_known0}
	\alpha^{(k)} = \OT_c(\mu^{(k)}, \nu^{(k)}) = \inner{c}{\pi^{(k)}}\,, \quad k \in \idn{K} \,.
\end{equation}
This setting is much more simple than that in \autoref{thm:ident_only_total} or \autoref{thm:ident_only_pot}, as we directly observe OT plans and therefore do not have to consider all possible ones. Recall that by \autoref{pr:ri} the observed OT plans $\pi^{(k)}$ potentially contain more information depending on where they lie in the corresponding polytope $\Pi(\pi^{(k)})$.

\begin{theorem} \label{thm:ident_total_plan}
	Suppose that we observe $K$ distinct pairs of marginals and the corresponding total costs \ref{enum:ot_cost} and OT plans \ref{enum:ot_plan}. Let $\{ u^{(\ell, k)} : \ell \in \idn{L_k} \}$, $k \in \idn{K}$, be as in \autoref{pr:ri}. Then, the underlying cost matrix $c$ is identifiable in $\R^{N \times M}$ if and only if the system
	\begin{align*}
		\inner{c}{v^{(k)}} &\geq \alpha^{(k)}, & v^{(k)} &\in \EP(\mu^{(k)}, \nu^{(k)}) \setminus \{ u^{(\ell,k)} : \ell \in \idn{L_k} \}, \\
		\inner{c}{u^{(\ell,k)}} &= \alpha^{(k)}, & \ell &\in \idn{L_k}, \; k \in \idn{K},
	\end{align*}
	has a unique solution. Note that if a solution exists, then we have existence of the underlying cost matrix.
\end{theorem}

\begin{remark}[Computational complexity] \label{rem:comp_complexity_total_plan}
	The system in \autoref{thm:ident_total_plan} is an LP with variable size $NM$ and $W \defeq \sum_{k=1}^{K} \abs{\EP(\mu^{(k)}, \nu^{(k)})}$ constraints. In particular, identifiability is thus linked to the uniqueness of solutions of LPs \cite{Mangasarian1979,Appa2002}. To check uniqueness of an LP, we can employ the PUFAS algorithm \cite{Appa2002}: Given a solution $c$ to the original LP, the same LP with adjusted objective is solved which yields another solution $c'$ to the original LP. We then have uniqueness if $c \neq c'$. Note that the adjusted objective is obtained by a straightforward transformation of $c$. Hence, checking uniqueness according to \autoref{thm:ident_total_plan} computationally boils down to solving two LPs with variable size $NM$ and $W$ constraints. As mentioned before, note that the number of extreme points can depend exponentially on $N$ and $M$, which then becomes computationally infeasible, in general.
\end{remark}

Dropping the inequalities in the system of \autoref{thm:ident_total_plan}, we get a sufficient condition for identifiability in terms of a linear equation:
\begin{equation} \label{eq:total_costs_plans_known_lin_system}
	\inner{c}{u^{(\ell,k)}} = \alpha^{(k)}, \qquad \ell \in \idn{L_k},\,k \in \idn{K}.
\end{equation}
From this we can get a simple identifiability criterion in terms of the OT plans. The proof idea and intuition behind the below theorem are very similar to the one for \autoref{thm:ident_only_plans} where we only observe the OT plans. In a sense, this is a stricter version on account of additionally observing the total costs.

\begin{theorem} \label{thm:ident_total_plan_eq}
    Suppose that we observe $K$ distinct pairs of marginals and the corresponding total costs \ref{enum:ot_cost} and OT plans \ref{enum:ot_plan}. Let $\{ u^{(\ell, k)} : \ell \in \idn{L_k} \}$, $k \in \idn{K}$, be as in \autoref{pr:ri}. Then, the following holds:
    \begin{enumerate}[a)]
        \item If the system $\{ u^{(\ell, k)} : k \in \idn{K},\, \ell \in \idn{L_k} \}$ contains $NM$ linearly independent vectors, then the cost matrix $c$ is identifiable in $\R^{N \times M}$.
        \item If $\coh \{u^{(\ell, k)} : \ell \in \idn{L_k} \} =\OPT_c(\pi^{(k)})$ for all $k \in \idn{K}$, then the converse holds.
    \end{enumerate}
\end{theorem}

\begin{proof}
    If the system contains $NM$ linearly independent vectors, then the solution $c$ to \eqref{eq:total_costs_plans_known_lin_system} is fully determined.

    For the converse, suppose that the system $\{ u^{(\ell, k)} : \ell \in \idn{L_k},\,k \in \idn{K} \}$ contains less than $NM$ linearly independent vectors. In particular, there exists $c' \neq 0$ in the kernel of the linear equations \eqref{eq:total_costs_plans_known_lin_system}. Define for $\tau > 0$ the cost matrix $c_\tau \defeq c + \tau c'$. We claim that for all $k \in \idn{K}$ there exists a $\tau(k) > 0$ such that for all $0 < \tau \leq \tau(k)$ it holds that $\pi^{(k)}$ is optimal w.r.t.\ the cost matrix $c_{\tau}$. Fix $k \in \idn{K}$ and assume that this is not the case, i.e., for all $\tau(k) > 0$ there exists a $0 < \tau \leq \tau(k)$ such that $\inner{\pi^{(k)}}{c_\tau} = \alpha^{(k)} > \OT_{c_\tau}(\mu^{(k)}, \nu^{(k)})$. Hence, we can construct sequences $\tau_s > 0$ and $\gamma_s \in \OPT_{c_{\tau_s}}(\pi^{(k)}) \cap \EP(\pi^{(k)})$ such that $\alpha^{(k)} > \inner{c_{\tau_s}}{\gamma_s}$ and $\tau_s \to 0$ for $s \to \infty$. As in the proof of \autoref{thm:ident_only_plans} it follows that there exist $\gamma \in \OPT_c(\pi^{(k)})$ and $s_0 \in \N$ such that $\gamma_s = \gamma$ for all $s \geq s_0$. If $\coh \{u^{(\ell, k)} : \ell \in \idn{L_k} \}=\OPT_c(\pi^{(k)})$, then $\gamma = u^{(\ell, k)}$ for some $\ell \in \idn{L_k}$ and thus $\inner{c'}{\gamma} = 0$. Hence, $\inner{c_\tau}{\gamma_s} = \inner{c_\tau}{\gamma}=\inner{c}{\gamma} = \alpha^k$, a contradiction. Therefore, for any $k \in \idn{K}$, the observed plan $\pi^{(k)}$ is optimal for the cost matrix $c_{\tau} = c + \tau c'$, with $0 \leq \tau\leq \min \{\tau(1), \ldots, \tau(K)\}$, so that $c$ is not identifiable. This concludes the proof.
\end{proof}

\begin{remark}[Computational complexity]
	The uniqueness condition of \autoref{thm:ident_total_plan_eq} a) can be checked again via Gaussian elimination and computation of the extreme points from \autoref{pr:ri}. Hence, the computational complexity mirrors the one in the case where we only observe the OT plans, see \autoref{rem:comp_complexity_plan}.
\end{remark}

\begin{remark}[Repeated pairs of marginals] \label{rem:observe_distinct_plans}
    In \autoref{thm:ident_total_plan} and \autoref{thm:ident_total_plan_eq} we assume that the $K$ observed pairs of marginals are distinct. If they are not, say $(\mu^{(1)}, \nu^{(1)}) = \ldots = (\mu^{(R)}, \nu^{(R)})$, then we can set the OT plan of this pair of marginals to $\frac{1}{R} \sum_{r=1}^R \pi^{(r)}$ and both theorems remain valid. Note that this joins the information of the individual OT plans $\pi^{(1)}, \ldots, \pi^{(R)}$ in the sense that the resulting face given in \autoref{pr:ri} is maximal in terms of the available information.
\end{remark}

\begin{remark}[Shift equivalence and total costs]
    Although shift equivalent \eqref{eq:shift_equiv} cost matrices induce the same OT plans, the total cost may differ. Hence, in the setting where we also observe the total costs, we can achieve identifiability in $\R^{N \times M}$ and not just $\shiftc^{N \times M}$.

    However, this setting can still be influenced by a stronger version of shift equivalency given a fixed number of marginals: Assume that there exist $a \in \R^N$ and $b \in \R^M$ such that for all $k \in \idn{K}$ it holds that
    \begin{equation*}
        \inner{a}{\mu^{(k)}} + \inner{b}{\nu^{(k)}} = 0\,.
    \end{equation*}
    Then, the shift $a \oplus b$ does not change the total cost for these specific $K$ marginals. Hence, when one of $a$ or $b$ is non-zero, then we have another cost matrix $c + a \oplus b \neq c$ that has the same OT plans and total costs for the given $K$ pairs of marginals. The condition above can be written as
    \begin{equation*}
        A a + B b = 0,
    \end{equation*}
    where $A \in \R^{K \times N}$ and $B \in \R^{K \times M}$ have the rows $\mu^{(1)}, \ldots, \mu^{(K)}$ and $\nu^{(1)}, \ldots, \nu^{(K)}$, respectively. This has a non-trivial solution $(a, b) \neq 0$ if and only if
    \begin{equation*}
        \ker A \neq \{ 0 \} \quad \text{or} \quad \ker B \neq \{ 0 \} \quad \text{or} \quad \im A \cap \im (-B) \neq \{ 0 \}\,.
    \end{equation*}
\end{remark}

We now illustrate \autoref{thm:ident_total_plan} and \autoref{thm:ident_total_plan_eq} with some examples. The first one shows that there always exist OT plans and marginals such that the underlying cost matrix $c$ is identifiable.

\begin{example}
	Let $K = NM$ and suppose that we observe the OT plans
	\begin{equation*}
		\pi^{(i + [j-1]M)} = \delta_i \otimes \delta_j \,, \quad i \in \idn{N}, \, j \in \idn{M}\,,
	\end{equation*}
    with corresponding total costs $\alpha^{(k)}$, $k \in \idn{K}$. Here, $\delta_i$ and $\delta_j$ denote the Dirac measures at $i$ and $j$, respectively, and $\delta_i \otimes \delta_j$ is their product measure. Then, each $\bm{\pi}^{(k)}$ is equal to the $k$-th standard basis vector $\bm{e}_k$ and $c$ is therefore identifiable with $c_{i,j} = \alpha^{(i + [j-1]M)}$.
\end{example}

\begin{example}
    We give an example that shows that for \autoref{thm:ident_total_plan_eq} the additional assumption of $\coh \{u^{(\ell, k)} : \ell \in \idn{L_k} \} =\OPT_c(\pi^{(k)})$ for all $k \in \idn{K}$ for the converse direction is crucial. Suppose that for $N = M = 2$ we observe the $K = 3$ OT plans
    \begin{equation*}
		\pi^{(1)} = \begin{bmatrix} 3/20 & 1/4 \\ 3/5 & 0 \end{bmatrix}\,, \quad \pi^{(2)} = \begin{bmatrix} 0 & 1/2 \\ 1/2 & 0 \end{bmatrix}\,, \quad \pi^{(3)} = \begin{bmatrix} 1/2 & 1/10 \\ 0 & 2/5 \end{bmatrix}\,,
	\end{equation*}
	with corresponding total costs $\alpha^{(1)} = 1/3$ and $\alpha^{(2)} = \alpha^{(3)} = 1$. As $K < NM$, we cannot use \autoref{thm:ident_total_plan_eq} to obtain identifiability of the underlying cost matrix in $\R^{N \times M}$. However, in this case \autoref{thm:ident_total_plan} yields identifiability and the underlying cost matrix is given by
	\begin{equation*}
		c = \begin{bmatrix} -1/3 & 7/3 \\ -1/3 & 7/3 \end{bmatrix}\,.
	\end{equation*}
\end{example}

Note for \autoref{thm:ident_total_plan_eq} it is crucial that we know that the underlying cost matrix $c$ exists. Else, it could happen that the observed OT plans are not optimal for the unique $c$ solving the system of linear equations \eqref{eq:total_costs_plans_known_lin_system}.

\begin{example} \label{ex:incons_c_lin_eq}
    Consider the $K = 4$ linearly independent transport plans
	\begin{align*}
		\pi^{(1)} &= \begin{bmatrix}
			0  & 1/4 \\ 1/2 & 1/4
		\end{bmatrix},
		& \pi^{(2)} &= \begin{bmatrix}
			1/3 & 0 \\ 1/3 & 1/3
		\end{bmatrix},
		\\ \pi^{(3)} &= \begin{bmatrix}
			1/3 & 1/3 \\  0 & 1/3
		\end{bmatrix},
		& \pi^{(4)} &= \begin{bmatrix}
			4/11 & 3/11 \\ 4/11 & 0
		\end{bmatrix}
	\end{align*}
	with total costs $\alpha^{(1)} = \alpha^{(2)} = 1$ and $\alpha^{(3)} = \alpha^{(4)} = 2$. Then, it follows that the unique solution to \eqref{eq:total_costs_plans_known_lin_system} is given by
	\begin{equation*}
		c = \begin{bmatrix}
			8/3 & 10/3 \\ 1/3 & 0
		\end{bmatrix}\,.
	\end{equation*}
    However, the total costs for the marginals of the above transport plans w.r.t.\ $c$ are given by $3/4$, $1$, $2$ and $19/11$, respectively.
\end{example}

\begin{figure}

	\centering

\tikzset{every picture/.style={line width=1pt}} 

\tikzset{every picture/.style={line width=1.25pt}} 

\begin{tikzpicture}[x=0.75pt,y=0.75pt,yscale=-1,xscale=1]

\draw  [color={rgb, 255:red, 0; green, 0; blue, 0 }  ,draw opacity=0 ][fill={rgb, 255:red, 219; green, 219; blue, 219 }  ,fill opacity=1 ] (242.6,93.28) -- (303.6,133.33) -- (276.38,155.1) -- (208.58,158.57) -- (173.44,132.83) -- cycle ;
\draw  [fill={rgb, 255:red, 0; green, 0; blue, 0 }  ,fill opacity=1 ] (240.51,93.28) .. controls (240.51,92.13) and (241.45,91.2) .. (242.6,91.2) .. controls (243.75,91.2) and (244.69,92.13) .. (244.69,93.28) .. controls (244.69,94.44) and (243.75,95.37) .. (242.6,95.37) .. controls (241.45,95.37) and (240.51,94.44) .. (240.51,93.28) -- cycle ;
\draw [color={rgb, 255:red, 0; green, 0; blue, 0 }  ,draw opacity=1 ] [dash pattern={on 0.84pt off 2.51pt}]  (219.16,77.81) -- (328.26,149.47) ;
\draw [color={rgb, 255:red, 0; green, 0; blue, 0 }  ,draw opacity=1 ][fill={rgb, 255:red, 255; green, 0; blue, 0 }  ,fill opacity=1 ] [dash pattern={on 0.84pt off 2.51pt}]  (153.97,143.79) -- (268.61,78.59) ;
\draw    (303.6,133.33) -- (276.38,155.1) -- (208.58,158.57) -- (173.44,132.83) ;
\draw  [fill={rgb, 255:red, 0; green, 0; blue, 0 }  ,fill opacity=1 ] (171.35,132.83) .. controls (171.35,131.68) and (172.29,130.74) .. (173.44,130.74) .. controls (174.59,130.74) and (175.53,131.68) .. (175.53,132.83) .. controls (175.53,133.99) and (174.59,134.92) .. (173.44,134.92) .. controls (172.29,134.92) and (171.35,133.99) .. (171.35,132.83) -- cycle ;
\draw  [fill={rgb, 255:red, 0; green, 0; blue, 0 }  ,fill opacity=1 ] (301.51,133.33) .. controls (301.51,132.17) and (302.44,131.24) .. (303.6,131.24) .. controls (304.75,131.24) and (305.69,132.17) .. (305.69,133.33) .. controls (305.69,134.48) and (304.75,135.42) .. (303.6,135.42) .. controls (302.44,135.42) and (301.51,134.48) .. (301.51,133.33) -- cycle ;
\draw  [fill={rgb, 255:red, 0; green, 0; blue, 0 }  ,fill opacity=1 ] (274.29,155.1) .. controls (274.29,153.95) and (275.22,153.01) .. (276.38,153.01) .. controls (277.53,153.01) and (278.47,153.95) .. (278.47,155.1) .. controls (278.47,156.26) and (277.53,157.19) .. (276.38,157.19) .. controls (275.22,157.19) and (274.29,156.26) .. (274.29,155.1) -- cycle ;
\draw  [fill={rgb, 255:red, 0; green, 0; blue, 0 }  ,fill opacity=1 ] (206.49,158.57) .. controls (206.49,157.41) and (207.42,156.48) .. (208.58,156.48) .. controls (209.73,156.48) and (210.67,157.41) .. (210.67,158.57) .. controls (210.67,159.72) and (209.73,160.66) .. (208.58,160.66) .. controls (207.42,160.66) and (206.49,159.72) .. (206.49,158.57) -- cycle ;

\draw (236.06,69.65) node [anchor=north west][inner sep=0.75pt]  [font=\small,color={rgb, 255:red, 0; green, 0; blue, 0 }  ,opacity=1 ,xscale=1.5,yscale=1.5] [align=left] {$\displaystyle \pi $};

\end{tikzpicture}
	\caption{The polytope $\Pi(\mu, \nu)$ with observed OT plan $\pi$. Visually, we see that a cost matrix $c$ is consistent with $\pi$ (up to a sign) if the hyperplane $\{ u \in \R^{NM} :  \inner{u}{c} =  \inner{\pi}{c}\}$ lies ``above'' the polytope, i.e., it is only allowed to touch the parts of $\Pi(\mu, \nu)$ that lie on the dotted lines. In particular, to enforce this behavior only $\pi$ and its two neighboring extreme points (i.e., the extreme points that lie on a 1-dimensional face containing $\pi$) need to be considered.} \label{fig:optimize_parts}
\end{figure}

Lastly, we note that in \autoref{thm:ident_total_plan} it is not necessary to enforce the constraint $\inner{c}{v^{(k)}} \geq \alpha^{(k)}$ for all $v^{(k)} \in \EP(\mu^{(k)}, \nu^{(k)}) \setminus \{ u^{(\ell, k)} : \ell \in \idn{L_k} \}$, $k \in \idn{K}$. Indeed, according to the next lemma it is enough to consider the extreme points of the biggest faces containing $\pi^{(k)}$, $k \in \idn{K}$. For geometric intuition, see \autoref{fig:optimize_parts}.

\begin{lemma} \label{lemma:smaller_set_geq}
	Let $c \in \R^{N \times M}$ be a cost function. Suppose that we observe marginals $(\mu, \nu) \in \Delta_N \times \Delta_M$ with optimal face $F \subseteq \OPT_c(\mu, \nu)$ and total cost $\alpha = \OT_c(\mu, \nu)$. Denote with $G_1, \ldots, G_{R_F}$ the sets containing the extreme points of the highest dimensional faces of $\Pi(\mu, \nu)$ (excepting $\Pi(\mu, \nu)$ itself) that contain $F$ and define the set
	\begin{equation*}
		G \defeq \begin{cases}
			\bigcup_{r=1}^{R_F} G_r \setminus F & R_F \neq 1\,, \\
			\{ w \} & R_F = 1\,, \text{ for any } w \in \EP(\mu, \nu) \setminus F \,.
		\end{cases}
	\end{equation*}
	Then, the set of cost functions $c$ that is consistent with the observed OT information is given by the solutions to the system
	\begin{equation*}
		\inner{c}{u} = \alpha,  \qquad \inner{c}{v} \geq \alpha, \qquad u \in F,\qquad v \in G.
	\end{equation*}
\end{lemma}
\begin{proof}
	Let $c$ be any solution to the above system, then we have to show that $\alpha = \OT_c(\mu, \nu)$. First, we do the following reduction to the dimension of $\Pi(\mu, \nu)$: Let $B \defeq \aff\{ \Pi(\mu, \nu) \} $ be the minimal affine space that $\Pi(\mu, \nu)$ lies in (i.e.\ its affine hull) and denote with $b \defeq \dim B$ its dimension. Consider the hyperplane $H \defeq \{ x \in \R^{NM} : \inner{c}{x} = \alpha \}$, then if $B \cap H = B$ we are done, else $B \cap H \subsetneq B$ which implies that $H' \defeq B \cap H$ is a $b$-dimensional hyperplane in $B$ with half-spaces $H'_{\geq} \defeq \{ x \in B : \inner{c}{x} \geq \alpha \}$ and $H'_{\leq} \defeq \{ x \in B: \inner{c}{x} \leq \alpha \}$. In particular, the hyperplane $H'$ is determined by $b - 1$ points that don't lie on a lower-dimensional hyperplane. Now, we consider the two cases $R_F = 1$ and $R_F > 1$ separately.

	If $R_F = 1$, this means that $F = G_1$ is maximum dimensional, i.e., it has dimension $b - 1$, and $H'$ is thus determined by $F$. Again using that $F$ is a maximum dimensional face, it follows that either $\Pi(\mu, \nu) \subseteq H'_{\geq}$ or $\Pi(\mu, \nu) \subseteq H'_{\leq}$. Due to the constraint $\inner{c}{w} \geq \alpha$ for one $w \in \EP(\mu, \nu) / F$, we must have $\Pi(\mu, \nu) \subseteq H'_{\geq}$ and thus $\alpha = \OT_c(\mu, \nu)$.

	If $R_F > 1$, we first consider the case that $\inner{c}{v} = \alpha$ for all $v \in G_i$ for some $i \in \idn{R_F}$. As $G_1, \ldots, G_{R_F}$ are maximum dimensional, it follows that there exists exactly one $i \in \idn{R_F}$ such that $\inner{c}{v} = \alpha$ for all $v \in G_i$. Hence, we can argue as in the $R_F = 1$ case that $\alpha = \OT_c(\mu, \nu)$. The other case is $\inner{c}{v} > \alpha$ for some $v \in G_i$ and all $i \in \idn{R_F}$. Suppose that $c$ is not consistent with the observed OT information, i.e., there exists $z \in \EP(\mu, \nu)$ with $\inner{c}{z} < \alpha$. Since by construction $\inner{c}{v} \geq \alpha$ for all $v \in \bigcup_{i=1}^{R_F} G_i$, note that $z \notin \bigcup_{i=1}^{R_F} G_i$. Fix a point $y$ in the relative interior of $F$ and consider the segment $[y, z] \defeq \{ \lambda y + (1 - \lambda) z : 0 \leq \lambda \leq 1 \} \subseteq \Pi(\mu, \nu)$ between $y$ and $z$. As $F$ has dimension strictly less than $b-1$ (else we would have $R_F = 1$), there exists a $2$-dimensional affine space $E$ that contains $[y, z]$ and is orthogonal to $F$. Then, $\Pi(\mu, \nu) \cap E$ is a $2$-dimensional polytope and we are in the setting of \autoref{fig:two_dim_polytope_proof}. Now, note that the segments $[y, w_1]$ and $[y, w_2]$ must belong to some $G_i$ as they are of maximal dimension. Hence, \autoref{fig:two_dim_polytope_proof} shows a contradiction: It holds that $\inner{c}{x} < \alpha$ for all $x \in [y, z] \setminus \{ y, z \}$ (the dotted line), but $\inner{c}{x} \geq \alpha$ for all $x \in \coh\{y, w_1, w_2 \}$ (the gray area).
\end{proof}

\begin{figure}
    \centering
    \includegraphics{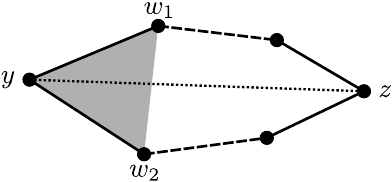}
    \caption{The $2$-dimensional polytope $\Pi(\mu, \nu) \cap E$ from the proof of \autoref{lemma:smaller_set_geq}.}
    \label{fig:two_dim_polytope_proof}
\end{figure}

\subsection{Identifiability from Optimal Transport Plans and One of the Optimal Potentials} \label{Sec:plans_one_pot}

The optimal potentials $(f^{(k)}, g^{(k)})$ from \ref{enum:ot_pot} come in pairs, in the following we call $f^{(k)}$ and $g^{(k)}$ the first and second optimal potential, respectively. In this section, we investigate the case where we observe the OT plans \ref{enum:ot_plan} as well as one, say the first, of the optimal potentials \ref{enum:ot_pot}.

As we do not observe the total costs \ref{enum:ot_cost}, note that we are again in a setting where $c$ is only identifiable up to shift equivalence \eqref{eq:shift_equiv}. However, knowing one of the optimal potentials restricts the form of the permissible shifts. Indeed, note that $(f, g)$ are optimal potentials w.r.t.\ the cost matrix $c$ if and only if $(f + a, g + b)$ are optimal potentials w.r.t.\ $c + a \oplus b$ for all $a \in \R^N$, $b \in \R^M$. Hence, if we observe the first optimal potentials, then we can identify $c$ only up to shifts of the form $0 \oplus b$, $b \in \R^M$, which we call \textit{marginal shift equivalence}. Denote with $\shiftc^{N \times M}_2$ the corresponding quotient space, i.e., $\R^{N \times M}$ modulo marginal shifts.

In the following, we will show that $c$ is identifiable in $\shiftc^{N \times M}_2$ if the supports $\supp \pi^{(k)}$, $k\in \idn{K}$ are sufficiently connected. To this end, denote $S_k \defeq \supp \pi^{(k)}$ and $S_{k,\ell} \defeq S_k \cap S_\ell$ for $k,\ell \in \idn{K}$. For $j \in \idn{M}$, we say that $S_k$ and $S_\ell$ are \textit{$j$-connected} if there exists sequences $r_1, \ldots, r_{q-1} \in \idn{K}$ with $r_0 \defeq k$, $r_q \defeq \ell$ and $i_1, \ldots, i_q \in \idn{N}$ such that $(i_v, j) \in S_{r_{v-1}, r_v}$ for all $v \in \idn{q}$. Using the primal-dual optimality criterion (\autoref{lemma:primal_dual_opt}), we will see that we can compute $g_j^{(k)}$ from $g_j^{(\ell)}$ if $S_k$, $S_\ell$ are $j$-connected. Now, the idea is to fix one $g^{(h)}$ that is $j$-connected to some $g^{(\ell)}_j$ that determines $c_{i,j} = f^{(\ell)}_i + g^{(\ell)}_j$, for all $j \in \idn{M}$. We summarize this in the following definition.

\begin{definition}
	Call the supports $S_1, \ldots, S_K$ \textit{connected (with center $S_h$)} if there exists an $h \in \idn{K}$ such that for all $(i, j) \in \idn{N} \times \idn{M}$ there exists an $\ell \in \idn{K}$ such that $(i, j) \in S_\ell$ and $S_h$, $S_\ell$ are $j$-connected.
\end{definition}

\begin{theorem} \label{thm:ident_plans_one_pot}
	Suppose that we observe $K$ pairs of marginals and the corresponding OT plans \ref{enum:ot_plan} and the first optimal potentials \ref{enum:ot_pot}. If $S_1, \ldots S_K$ are connected, then the cost matrix $c$ is identifiable in $\shiftc_2^{N\times M}$.
\end{theorem}
\begin{proof}
	Suppose that $S_1, \ldots, S_K$ are connected with center $S_h$. Denote with $g^{(k)}, k \in \idn{K}$, the second optimal potentials which we do not observe. We can assume that $g^{(h)} \equiv 0$ by marginal shift invariance. Indeed, otherwise consider the cost matrix $c + 0 \oplus [-g^{(h)}]$ which has the same OT plans and first optimal potentials as $c$, but the second optimal potentials are given by $g^{(k)} - g^{(h)}$, $k\in\idn{K}$. By the primal-dual optimality criterion (\autoref{lemma:primal_dual_opt}), it holds for all $k, \ell \in \idn{K}$ that
	\begin{equation*}
	f^{(k)} \oplus g^{(k)} = c = f^{(\ell)} \oplus g^{(\ell)} \quad \text{on } S_{k,\ell}\,.
\end{equation*}
	In particular, for $(i, j) \in S_{k,\ell}$ we have
	\begin{equation*}
		g^{(k)}_j = g^{(\ell)}_j + \tau^{(k,\ell)}_i\,,
	\end{equation*}
	where $\tau^{(k,\ell)}_i \defeq f^{(\ell)}_i - f^{(k)}_i$ is known and independent of $j$. Hence, if we know $g_j^{(\ell)}$ and $S_k$, $S_\ell$ are $j$-connected, then we can also determine $g_j^{(k)}$. By connectivity, for all $(i, j) \in \idn{N} \times \idn{M}$ there exists an $\ell \in \idn{K}$ such that $(i, j) \in S_\ell$ and $g^{(\ell)}_j$ is determined from $g^{(h)} \equiv 0$. In particular, this forces $c_{i,j} = f^{(\ell)}_i + g^{(\ell)}_j$ by primal-dual optimality (\autoref{lemma:primal_dual_opt}).
\end{proof}

\begin{remark}[Computational complexity]
	To determine if $S_1, \ldots, S_K$ are connected with center $S_h$, for each $j \in \idn{M}$ we first build the $j$-connectivity graphs $G_j$, i.e., with vertices $\idn{K}$ and an edge between $k, \ell \in \idn{K}$ if there exists an $i \in \idn{N}$ such that $(i,j) \in S_{k,\ell}$, and calculate their connected components (e.g.\ using breadth-first search). In the second step, for each $h \in \idn{K}$ and $j \in \idn{M}$ we find the connected component $G_j(h)$ of $G_j$ that contains $h$ and check whether each $i \in \idn{N}$ is contained in some $S_\ell$ for $\ell \in G_j(h)$. If this is the case, then the supports are connected with center $S_h$. Note that both steps take $\bigO(K^2NM)$ operations.
\end{remark}

\subsection{Identifiability from Full Information}\label{Sec:FullInf}

Finally, we consider the simplest case where we observe the full OT information \hyperlink{eq:obs}{(OBS)}, i.e., the total costs \ref{enum:ot_cost}, OT plans \ref{enum:ot_plan} and optimal potentials \ref{enum:ot_pot}.

\begin{theorem} \label{thm:ident_plan_pot}
	Suppose that we observe $K$ pairs of marginals and the corresponding total costs \ref{enum:ot_cost}, OT plans \ref{enum:ot_plan} and optimal potentials \ref{enum:ot_pot}. Then, $c$ is identifiable in $\R^{N \times M}$ if and only if for all $(i, j) \in \idn{N} \times \idn{M}$ there exists a $k \in \idn{K}$ such that $\pi^{(k)}_{i,j} > 0$. In other words, $c$ is uniquely determined on $\bigcup_{k=1}^K \supp \pi^{(k)}$.
\end{theorem}
\begin{proof}
	\autoref{lemma:primal_dual_opt} implies that the cost matrix $c$ is uniquely determined on $\bigcup_{k=1}^K \supp \pi^{(k)}$. Suppose that this support is not full, i.e., there exists a pair $(r, s) \in \idn{N} \times \idn{M}$ with $(r, s) \notin \bigcup_{k=1}^K \supp \pi^{(k)}$. Consider the matrix $\tilde{c} \in \R^{N \times M}$ that is equal to $c$ except $\tilde{c}_{r,s} = c_{r,s} + 5$. Then, note that $c \leq \tilde{c}$ and thus $(f^{(k)}, g^{(k)}) \in \Phi_{\tilde{c}}$ for all $k \in \idn{K}$. Hence, \autoref{lemma:primal_dual_opt} yields that $\pi^{(k)}$ and $(f^{(k)}, g^{(k)})$ are also optimal w.r.t.\ $\tilde{c}$ for all $k \in \idn{K}$. Consequently, in this case $c$ is not identifiable.
\end{proof}

\begin{remark}[Computational complexity]
    In the setting of \autoref{thm:ident_plan_pot}, the cost matrix $c$ can be directly computed on $\bigcup_{k=1}^K \supp \pi^{(k)}$ from the optimal potentials. More specifically, it holds that $c = \max \{ f^{(k)} \oplus g^{(k)} : k \in \idn{K} \}$ where the maximum is taken component-wise. In particular, identifiability can be checked in $\bigO(KNM)$ operations.
\end{remark}

\subsection{Identifiability Under Structural Assumptions on the Cost Matrix} \label{Sec:struct}

We will investigate the impact of structural assumptions on $c$ on identifiability. Suppose that $N = M$. Often, a reasonable assumption is $c_{i,i}=0$, for all $i \in \idn{N}$, which means that the cost of transporting a point to itself is zero. Additionally, we assume that $c$ is symmetric (e.g.\ when resulting from a metric) and thus belongs to the set
\begin{equation} \label{eq:sym_mat}
   \sym_N^0 \coloneqq \{ c \in \R^{N \times N}:\ \text{$c_{i,i}=0$ and $c_{i,j}=c_{j,i}$ for all $i,j \in \idn{N}$}\}\,.
\end{equation}
Denoting with $\utr{c}$ and $\ltr{c}$ the vectorized strict upper and lower triangular matrices of $c$, respectively, we have that $c \in \sym_N^0$ is fully determined by $\utr{c}$ (or $\ltr{c}$). In particular, we only have $(N-1)N/2$ free variables instead of the full $N^2$. As such, we need less information to have identifiability of $c$. Clearly our results that rely on the unique solvability of LP(s) (e.g.\ \autoref{thm:ident_only_total}, \autoref{thm:ident_only_pot}, \autoref{thm:ident_total_plan}) remain true for this reduced problem. In particular, we obtain the following corresponding result to \autoref{thm:ident_total_plan_eq}.

\begin{theorem} \label{thm:ident_total_plan_eq_sym}
    Suppose that we observe $K$ distinct pairs of marginals and the corresponding total costs \ref{enum:ot_cost} and OT plans \ref{enum:ot_plan}. Let $\{ u^{(\ell, k)} : \ell \in \idn{L_k} \}$, $k \in \idn{K}$, be as in \autoref{pr:ri}. Then, the following holds:
    \begin{enumerate}[a)]
        \item If the system $\{ [\utr{} + \ltr{}] \vc{u}^{(\ell, k)} : \ell \in \idn{L_k},\,k \in \idn{K} \}$ contains $(N-1)N/2$ linearly independent vectors, then the cost matrix $c$ is identifiable in $\sym_N^0$.
        \item If $\coh \{u^{(\ell, k)} : \ell \in \idn{L_k} \} =\OPT_c(\pi^{(k)})$ for all $k \in \idn{K}$, then the converse holds.
    \end{enumerate}
\end{theorem}

Note that if we only observe the OT plans, we do not have the problem of shift invariance anymore. Indeed, for $a,\, b \in \R^N$ it follows that $a \oplus b \in \sym_N^0$ if and only if $a \oplus b = 0$. Hence, it holds that $\sym^0_N = \sym^0_N / { \shifteq } $. Thus, we obtain the following specialization of \autoref{thm:ident_only_plans}.

\begin{theorem} \label{thm:ident_only_plans_sym}
    Suppose that we observe $K$ distinct pairs of marginals and the corresponding OT plans \ref{enum:ot_plan}. Let $\{ u^{(\ell, k)} : \ell \in \idn{L_k} \}$, $k \in \idn{K}$, be as in \autoref{pr:ri}. Then, the following holds:
    \begin{enumerate}[a)]
        \item If it holds for $S \in \N_0$ that
    \begin{equation*}
 	 \dim ( \spanop \{ [\utr{} + \ltr{}] (\bmpi^{(k)} - \bmu^{(\ell,k)}) : \ell \in \idn{L_k},\,k \in \idn{K} \} ) = (N-1)N/2 - S,
 \end{equation*}
    then, the cost matrix $c$ is identifiable in $\sym_N^0$ up to the span of $S$ linearly independent vectors.
    \item If $\coh \{u^{(\ell, k)} : \ell \in \idn{L_k} \} =\OPT_c(\pi^{(k)})$ for all $k \in \idn{K}$, then the converse holds.
    \end{enumerate}
\end{theorem}
\begin{proof}
    Since $c \in \sym_N^0$, it holds for all $x \in \R^{N \times N}$ that
    \begin{equation*}
        \inner{c}{x} = \inner{ \utr{c} }{ [ \utr{} + \ltr{} ] \vc{x}}\,.
    \end{equation*}
    Hence, as in the proof of \autoref{thm:ident_only_plans}, we conclude that
    \begin{equation*}
        \utr{c} \in \left[ \spanop \left\{ [\utr{} + \ltr{}] (\bmpi^{(k)} - \bmu^{(\ell,k)}): \ell \in \idn{L_k},\, k \in \idn{K} \right\} \right]^\perp\,.
    \end{equation*}
    The rest follows analogously to \autoref{thm:ident_only_plans} by considering $c_\tau = c + \tau c' \in \sym_N^0$ where $c' \in \sym_N^0$ is linearly independent of $c$ and
    \begin{equation*}
        \utr{c'} \in \left( \spanop \left\{ [\utr{} + \ltr{}] (\bmpi^{(k)} - \bmu^{(\ell,k)}): \ell \in \idn{L_k},\, k \in \idn{K} \right\} \right)^\perp \,. \qedhere
    \end{equation*}
\end{proof}

\section{Statistical Application: Estimation of the Cost} \label{Sec:Estimation}

If we receive $NM$ linearly independent OT plans \ref{enum:ot_plan} and the corresponding total costs \ref{enum:ot_cost}, we would like to solve \eqref{eq:total_costs_plans_known0}. This can be achieved by computational linear algebra.  However, in practice, the values of the OT plans and the total costs may be subjected to some random noise, so that the observation is no longer \eqref{eq:total_costs_plans_known0}, but
\begin{equation} \label{eq:pi_known_noisy}
	y^{(k)} \coloneqq \alpha^{(k)}+\eps^{(k)} = \inner{c}{\pi^{(k)}} + \eps^{(k)} \,, \qquad k \in \idn{K},
\end{equation}
where $\{\eps^{(k)} \}_{k=1}^K$ is a sequence of independent centered random variables with common variance $\sigma^2 > 0$ (see \autoref{rem:gen_lin_mod} for generalizations). Upon defining the following matrix and vectors
\begin{align*}
	P_K &\defeq \begin{bmatrix}
		\bmpi^{(1)} & \ldots & \bmpi^{(K)}
	\end{bmatrix}^\transp,&  \bm{Y}_K &\defeq [ y^{(k)} ]_{k=1}^K, \\
	 \bmalpha_K &\defeq [\alpha^{(k)}]_{k=1}^K, & \bm{\eps}_K &\defeq [\eps^{(k)}]_{k=1}^K\,,
\end{align*}
the display \eqref{eq:pi_known_noisy} can be rewritten as
\begin{equation} \label{eq:pi_known_mat_noisy}
	P_K \bm{c}  = \bm{Y}_K = \bm{\alpha}_K + \bm{\eps}_K,
\end{equation}
where $\bm{\eps}_K$ is a centered random vector with covariance matrix $\sigma^2 \id$. This reveals \eqref{eq:pi_known_mat_noisy} as a statistical linear model, and we can make use of its well-developed theory. One approach to solve \eqref{eq:pi_known_mat_noisy} is to employ the least squares estimator (which is statistically efficient when the error is normally distributed):
\begin{equation} \label{eq:LSE}
	\hat{c}_K \in \argmin_{c \in  \R^{N \times M}} \norm{\bm{Y}_K - P_K \bm{c}}^2_2\,.
\end{equation}
If $NM$ elements of $\{\bm{\pi}^{(k)}\}_{k=1}^K$ are linearly independent, the matrix $P_K^\transp P_K$ is invertible and the solution can be expressed as
\begin{equation*}
	\hat{\bm{c}}_K = (P_K^\transp P_K)^{-1} P_K^\transp \bm{Y}_K\,.
\end{equation*}

\begin{theorem}\label{thm:CLT}
	Assume that there exists a solution $c$ of \eqref{eq:total_costs_plans_known0} and that we observe for $K = 1, 2, \ldots$ a sequence of noisy observations $\bm{Y}_K$ in \eqref{eq:pi_known_noisy}. Suppose that there exists a positive definite matrix $Q \in \R^{NM \times NM}$ such that
	\begin{equation*}
		\frac{1}{K} P_K^\transp P_K \to Q \qquad \text{for } K \to \infty\,.
	\end{equation*}
	Then, the least squares estimator $\hat{c}_K$ from \eqref{eq:LSE} is strongly consistent, i.e.,
	\begin{equation*}
		\hat{\bm{c}}_K \asconv \bm{c} \qquad \text{for } K \to \infty\,.
	\end{equation*}
	If additionally there is a $\delta > 0$ such that the $(2+\delta)$-th moment of $\eps^{(k)}$ exists and is uniformly bounded for all $k=1,2,\ldots$, then it follows that
	\begin{equation*}
		\sqrt{K}(\hat{\bm{c}}_K -\bm{c}) \wconv \calN({\bf 0}, \sigma^2 Q^{-1}) \qquad \text{for } K \to \infty \,,
	\end{equation*}
	where $\wconv$ denotes convergence in distribution and $\calN(\mu, \Sigma)$ is a $NM$-dimensional normal distribution with expectation $\mu$ and covariance $\Sigma$.
\end{theorem}
\begin{proof}
	Follows from \cite[Theorem~2.12 and Exercise~5.12]{White2001}.
\end{proof}

Based on the above limit law, we can derive the following asymptotic $(1-\alpha)$-confidence ellipsoid for $\bm{c}$,
\begin{equation*}
	\biggl\{ c \in \R^{N \times M} : [\hat{\bm{c}}_K -\bm{c}]^\transp Q [\hat{\bm{c}}_K -\bm{c}] \leq \frac{\sigma^2}{K} \chi^2_{NM, 1-\alpha} \biggr\}\,,
\end{equation*}
where $\chi^2_{NM, 1-\alpha}$ is the $(1-\alpha)$-quantile of the chi-squared distribution with $NM$ degrees of freedom.

In a similar fashion this can be extended to different estimators such as the $L^1$ minimizer of \eqref{eq:pi_known_mat_noisy}, see e.g.\ \cite{Carroll1982}, and different measurement errors $\bm{\eps}_K$, e.g., for certain dependency structures, see e.g.\ \cite{Wu2007}.

\begin{remark} \label{rem:gen_lin_mod}
	Our linear model is based on \eqref{eq:total_costs_plans_known0} and not \eqref{eq:total_costs_plans_known_lin_system}. If we want to use the additional information provided by \autoref{pr:ri}, we have a more complicated linear model. Indeed, then the noise is not i.i.d.\ anymore but coupled, i.e., multiple observations have the exact same noise. This can be dealt with similarly, but now the least squares estimator in \eqref{eq:LSE} has to be replaced by the generalized least squares estimator \cite{Amemiya1985,Rao2008}.
\end{remark}

\subsection{Identifiability in High-Dimensional Linear Models: Sparsity}\label{Sec:EstSparse}

We again consider the noisy model \eqref{eq:pi_known_noisy} and assume that $c$ is sparse. In the usual sense, this takes the form of $c_{i,j} = 0$ for many pairs of indices $(i, j) \in \idn{N} \times \idn{M}$. As this means that transport between $x_i$ and $y_j$ is for free and hence often not a reasonable assumption, we utilize a shifted version tailored to OT. More specifically, for a predefined value $b_0 \in \R$, we assume that the set $S \defeq \{ (i,j) \in \idn{N} \times \idn{M} : c_{i,j} \neq b_0 \}$ has small cardinality $s \ll NM$, i.e., that the transport cost between many points is equal. In particular, this means that the exact choice of transport between these pairs of points is irrelevant. In this setting, we can make use of this type of sparsity to effectively estimated $c$ even when $NM$ is large compared to the number of observations $K$. We illustrate estimation of $c$ with a corresponding modification of the LASSO estimator \cite{Tibshirani1996}
\begin{equation} \label{eq:LASSO}
	\hat{c}_K \in \argmin_{c \in \R^{N \times M}} \norm{\bm{Y}_K - P_K \bm{c}}^2_2 + \lambda_K \norm{\bm{c} - b_0}_1\,.
\end{equation}
To be able to give error bounds for the LASSO estimator, we require that $P_K$ satisfies the restricted eigenvalue condition over $S$ with parameters $\kappa > 0$ (cf.\ \cite{Wainwright2019,Buehlmann2011}). For $x \in \R^{NM}$ denote with $x_S$ and $x_{S^c}$ the restriction of $x$ to the index sets $S$ or $S^c$. Then, $P_K$ satisfies the restricted eigenvalue property (REP) over $S$ with parameter $\kappa > 0$ if for all $x \in \R^{NM}$ with $\norm{x_{S^c}}_1 \leq 3 \norm{x_{S}}_1$ it holds that
\begin{equation} \label{eq:restr_ev_cond} \tag{REP}
	\frac{1}{K} \norm{P_K x}_2^2 \geq \kappa \norm{x}_2^2\,.
\end{equation}
Naturally, \eqref{eq:restr_ev_cond} is satisfied if the observed OT plans contain $\{ \delta_i \otimes \delta_j \mid (i, j) \in S \}$, i.e., $P_K$ is the identity matrix when restricting to $S$. This means that we observe all the non-sparse cost coefficients directly (up to noise). However, more generally it is not clear to us what conditions on the transport plans imply \eqref{eq:restr_ev_cond}.

Under \eqref{eq:restr_ev_cond} (and related conditions), error bounds for the LASSO and related estimators have been the topic of a rich body of literature, see e.g.\ \cite[]{Giraud2021,Wainwright2019,Buehlmann2011} for recent monographs and further references. Exemplarily, we assume that the noise is normally distributed, then we obtain the following error bound from \cite[Example~7.14]{Wainwright2019}.

\begin{theorem}\label{thm:BoundsLasso}
	Assume that $P_K$ satisfies \eqref{eq:restr_ev_cond} over $S$ with parameter $\kappa > 0$ and that the noise $\eps^{(1)}, \ldots, \eps^{(K)}$ is i.i.d.\ zero-mean normally distributed with variance $\sigma^2 > 0$. Then, it holds for any $\delta > 0$ for the LASSO type estimator $\hat{\bm{c}}_K$ from \eqref{eq:LASSO} with regularization parameter
	\begin{equation*}
		\lambda_K = 2 \sigma \left[ \sqrt{\frac{2 \log N + 2 \log M}{K}} + \delta \right]
	\end{equation*}
	with probability at least $1 - 2e^{-K\delta^2 / 2}$ that
	\begin{equation*}
		\norm{\hat{\bm{c}}_K - \bm{c}}_2 \leq \frac{6 \sigma}{\kappa} \sqrt{s} \left[ \frac{2 \log N + 2 \log M}{K} + \delta \right]\,.
	\end{equation*}
\end{theorem}

\section*{Acknowledgements}

We are grateful to two anonymous referees for their helpful comments, which led to various improvements.

The research of M.\ Groppe was funded by the Deutsche Forschungsgemeinschaft (DFG, German Research Foundation) RTG 2088 ``\emph{Discovering structure in complex data: Statistics meets Optimization and Inverse Problems}'' and A.\ Munk acknowledges support of the DFG RU 5381 ``\emph{Mathematical Statistics in the Information Age --- Statistical Efficiency and Computational Tractability}''.

\printbibliography[title = References]

\end{document}